\documentclass[12pt]{amsart}

\usepackage[
text={440pt,575pt},
headheight=9pt,
centering
]{geometry}
\usepackage{graphicx}
\usepackage{amsmath, amssymb, amsthm, enumerate, bm}
\usepackage{mathtools}
\usepackage{dsfont}
\usepackage{float}

\allowdisplaybreaks 

\usepackage{mathptmx}

\usepackage{caption}
\usepackage{tabularx}
\newcolumntype{x}[1]{!{\centering\arraybackslash\vrule width #1}}
\usepackage{booktabs}
\usepackage{xcolor}
\usepackage{hyperref}
\definecolor{darkblue}{RGB}{0,0,160}
\hypersetup{
colorlinks,%
citecolor=darkblue,%
filecolor=black,%
linkcolor=darkblue,%
urlcolor=darkblue
}

\usepackage{todonotes}
\newcolumntype{E}{>{\footnotesize \selectfont}l<{}} 
\newcolumntype{F}{>{\small\selectfont}l<{}}

\newcommand{\N}{\mathbb{N}}

\newcommand{\R}{\mathbb{R}}

\newcommand{\V}{\mathcal{V}}

\DeclareMathOperator{\pnt}{\raise 0.5mm \hbox{\large\bf.}}

\newtheoremstyle{thm}{}{}
     {\em}
     {}
     {\bf}
     {.}
     {0.5em}
     {\thmname{#1}\thmnumber{ #2}\thmnote{ #3}}

\newtheoremstyle{def}{}{}
     {\rm}
     {}
     {\bf}
     {.}
     {0.5em}
     {\thmname{#1}\thmnumber{ #2}\thmnote{ #3}}

\theoremstyle{thm}

\newtheorem{thm}{Theorem}[section]
\newtheorem{lem}[thm]{Lemma}
\newtheorem{cor}[thm]{Corollary}

\newtheorem{conj}[thm]{Conjecture}

\theoremstyle{def}

\newtheorem{defi}[thm]{Definition}
\newtheorem{rem}[thm]{Remark}
\newtheorem{exa}[thm]{Example}

\newtheorem{Not}[thm]{Notation}
\newtheorem{Req}[thm]{Requirement}

\let\epsilon=\varepsilon
\let\phi=\varphi
\let\kappa=\varkappa

\author{Mandala von Westenholz}
\address{Osnabr\"uck University\\ Osnabr\"uck, Germany}
\email{mvonwestenho@uos.de}

\title[volume power functional]{
Covariance matrices of volume power functionals of random simplicial complexes -- an asymptotic analysis}



\begin{document}
\begin{abstract} 
This work analyzes and compares the asymptotic properties of the covariance matrices of vectors of volume power functionals of random Vietoris-Rips complexes, as the intensity of the underlying homogeneous Poisson point process grows.  Several key results are established which, in particular, generalize well-known facts on random graphs. Findings regarding rank, definiteness, determinant, eigenspaces, and related decompositions are presented within three distinct regimes. Moreover, we derive stochastic applications of these algebraic properties, leading to interesting results for vectors of volume power functionals.
\end{abstract}

\maketitle

\section{Introduction}  
Simplicial complexes have gained significant importance in recent research in stochastic geometry and related areas like in the field of Topological Data Analysis (see, e.g., \cite{Carlsson}, \cite{CarlssonZomorodian}, \cite{EdelsbrunnerHarer}, \cite{EdelsbrunnerLetscher},  \cite{Grygiereketal2020}, \cite{Wasserman}). 
In stochastic geometry, the use of simplicial complexes led to significant contributions to limit theorems and homology (see, e.g., \cite{BobrowskiKahle}, \cite{Kahle}, \cite{Linial}, \cite{Meshulam}) and they are a fundamental topic in combinatorics (see, e.g., \cite{
Fomin, Joswig}). 
They also have numerous applications addressing cross-disciplinary real-world problems, such as high-dimensional pattern mining, network intrusion detection and protein interaction networks (see, for instance, \cite{Atzmueller}, \cite{Atzmueller1}, \cite{Ernesto}).\\
There exist several possibilities to build a simplicial complex out of a point cloud, as the \textit{Vietoris-Rips} (see \cite{AkinwandeReitzner}) and the \textit{C\v{e}ch complex }(see \cite{EdelsbrunnerHarer}) construction. In this paper Vietoris-Rips complexes will always be considered in a random setting. A Vietoris-Rips complex is then constructed via a given set of points $X\subseteq \R^d$. Its $k$-simplices are the point sets of cardinality $k+1$, where all points have pairwise distance at most $\delta$ regarding a fixed metric. Here, the considered random points of the simplicial complex originate from a \emph{Poisson point process} $\eta$ (see \cite{LastPenrose}).
In this paper we study a sequence of such Vietoris-Rips complexes $\mathcal{R}(\eta_t, \delta_t)$ with homogeneous Poisson point processes $\eta_t$ with \emph{intensity} $t$ in a convex compact set $W \subseteq \R^d$  for distances $\delta_t \rightarrow 0$ for $t \rightarrow \infty$. \\
These assumptions lead to the following situations: A set $\{x_1, \ldots, x_{k+1}\} $ fulfills for each pair $x_i\neq x_j\in \eta_t$, where $i \neq j \in \{1, \ldots, k+1\}$, that  
$
\| x_i -x_j \| \leq \delta_t,
$  if and only if it is a $k$-simplex in the random simplicial complex $\mathcal{R}(\eta_t, \delta_t)$. In dependence on $t\delta_t^d$ the following regimes arise:

\begin{itemize}
\item \emph{subcritical (sparse) regime}  for $t\delta_t^d \xrightarrow{t \rightarrow \infty} 0$,
\item \emph{critical (thermodynamic) regime}  for $t\delta_t^d \xrightarrow{t \rightarrow \infty} c \in \R_{> 0}$,
\item \emph{supercritical (dense) regime} for $t\delta_t^d \xrightarrow{t \rightarrow \infty}  \infty$.
\end{itemize} 
In the presented work, the goal is to examine the \textit{volume power functionals} \begin{align}
\label{Eq: volume power-functional}
\V_k^{(\alpha)}=\frac{1}{(k+1)!} \sum_{(x_0,\ldots, x_k) \in \eta_{t,\, \neq}^{k+1}} \Delta_{\delta_t}[x_0, \ldots, x_k]^{\alpha},
\end{align} which are \textit{Poisson functionals} (see, e.g., \cite{Peccati, Reitzner, SchneiderWeil, Yoshifusa}) through the asymptotic behavior of the covariance matrix $\Sigma_n$ of a vector of such volume power functionals $(\mathcal{V}^{(\alpha_1)}_{k_1}, \ldots, \mathcal{V}^{(\alpha_n)}_{k_n})$ with respect to the constructed simplicial complexes. 

This article is in some parts based on results of the asymptotic investigation from \cite{AkinwandeReitzner} in which the authors discuss several algebraic properties of such a vector.
Moreover, it is an extension and is motivated by some results from \cite{ReitznerRoemer}, where the authors consider the special case $\mathcal{V}^{(\alpha)}_{1}$, which is called the \textit{length power functional}. 

There are further interesting special cases of volume power functionals. For instance, $f$-vectors of simplicial complexes are fundamental in various fields, including combinatorics and stochastics (see, e.g., \cite{Barany, Buchta, McMullen}). These vectors correspond to the specific case of volume power vectors where $\alpha_i = 0$ for all $i$. Details are outlined in the tables in Section \ref{Kap: Main results}. Additionally, an $h$-vector (see, e.g., \cite{McMullen}, \cite{Stanley}), as an equivalent information to a given $f$-vector, is of great interest and can be studied by our approach. This yields a novel contribution to the existing literature in our context. We will discuss this further in Section \ref{Sec: stoch. appl.}.

This paper is structured in the following way: In Section \ref{Section: preliminiaries} we introduce the basic notions. To get an overview we display in Section \ref{Kap: Main results} all main results concerning algebraic aspects like rank, definiteness, determinant, eigenspaces and selected decompositions in two tables. Due to the different matrix structures in the regimes, different results occur in each of them. The results are then picked up for applications in Section \ref{Sec: stoch. appl.}. To supplement the studies from Akinwande--Reitzner \cite{AkinwandeReitzner}  we expand the analysis of the covariance matrix of a volume power vector $\Sigma_n$ in Section \ref{Section: volume power functional} as extensive as possible. The last Section \ref{Section: Outlook} gives an overview on open questions and on possible approaches to solve them. \\
\textbf{Acknowledgments.} The author is grateful to Matthias Reitzner and Tim Römer from Osnabrueck University for inspiring discussion on the subject of the paper. The research was partially supported by the Lower Saxony Ministry of Science and Culture (MWK), through the zukunft.niedersachsen program of the Volkswagen Foundation (``Joint-Lab Artificial Intelligence \& Data Science'').

\section{Preliminaries} \label{Section: preliminiaries}
In this section we present a framework for the upcoming examinations. 
Here, a considered simplicial complex $\mathcal{R}(\eta_t, \delta_t)$ with $\delta_t \xrightarrow{t \to \infty}{0} $ arises through a Poisson point process, which is always given through \[\eta_t \colon \Omega  \rightarrow \mathbf{N}(W), \,\omega \mapsto \sum_{i=1}^{N(\omega )} \delta_{X_i(\omega)}.\] The set $W \subseteq \R^d$ is convex and compact and its volume is fixed as $V(W)=1$. The $X_i \colon \Omega \rightarrow W$ for $i \in [N(\omega)]$ and $N \colon \Omega \rightarrow \N$ are random variables. We denote by $\mathbf{N}(W)$ the set of finite counting measures on $W$ and $[n]$ gives the set $\{1, \ldots, n\}$ for $n \in \mathbb{N}$.  
 
Now we introduce the main notations that are used in the upcoming parts.
\begin{Not} 
\label{Not: PPP} \
\begin{enumerate}
\item Vertices $\{x_i,\ldots,  x_j\} $ are shortened as $ \{x_l\}_{l=i}^j$.
\item  $\eta_t(\omega)$ denotes the counting measure  $\sum_{i=1}^{N(\omega)} \delta_{X_i(\omega)}$ as well as the support of the counting measure, which is the set of points $\{X_1(\omega), \ldots, \,X_{N(\omega)}(\omega)\}$.
\item 
Let $A\in \mathcal{B}(W)$ be a Borel set. 
Then $\eta_t(A)=\eta_t(\omega)(A)=\sum_{i=1}^N \delta_{x_i}(A)$ gives the number of random points in $A$. Furthermore, $\eta_t\cap A=\eta_t(\omega)\cap A=
\{x_1,\ldots, \,x_n\}\cap A$ gives the explicit points of the support of $\eta$ that are lying in $A$.
\item $\eta_{t, \,\neq}^m (\omega)= \{ (x_{1}, \ldots, \,x_{m}) \mid 
j_1, \ldots, \,j_m\in [N(\omega)] 
\text{ and } x_i =X_{j_i}(\omega)  \text{ pairwise distinct} \}$ 
is the set of all $m$-tuples of pairwise distinct points in $\eta_t$. 
\item We abbreviate realizations in the following way: $\eta_t=\eta_t(\omega)$, $\eta_{t, \,\neq}^m (\omega) = \eta_{t, \,\neq}^m$,
$ x_i=X_i(\omega)$, and $N=N(\omega)$.
\item 
$\Delta_s [x_0, \ldots, x_k]$ is the $k$-dimensional volume of the convex hull of points $x_0, \ldots, x_k$, if all edge lengths are at most $s$ and $0$ otherwise. For the case $k > d$, since the $k$-dimensional volume is always zero, we define $\Delta_s [x_0, \ldots, x_k]^0$ to be the indicator function which equals $1$, if all pairwise distances are at most $s$, and $0$ otherwise.
\item We abbreviate 
\begin{align} \label{Eq: mu k}
    \mu_k^{(\alpha)} := \int_{(B^d ) ^k }\Delta_1[0, \{ x_l\}^k _{l=1}] ^{\alpha} \text{d}x_1 \ldots \text{d}x_k
\end{align}
and 
\begin{align} \label{Eq: mu ki kj}
    \mu_{k_1,k_2:m}^{(\alpha_1,\alpha_2)}:= \int_{(B^d)^{k_1+k_2+1-m}}\Delta_1[0,\{x_l\}^{k_1}_{l=1}]^{\alpha_1}\Delta_1[0,\{x_l\}^{k_1+k_2-m+1}_{l=k_1-m+2}]^{\alpha_2} \text{d}x_1 \ldots \text{d}x_{k_1+k_2+1-m}.
\end{align} In the case $k > d$, these definitions are only meaningful for $\alpha = 0$, in which case it serves as an indicator for whether all pairwise distances are bounded by $1$.
\end{enumerate}  

\end{Not}

In this paper we investigate the volume power functional $\V_k^{(\alpha)}$ (see (\ref{Eq: volume power-functional})), which sums up the $\alpha$-th powers of the $k$-dimensional volume of the $k$-simplices in the Vietoris-Rips complex $\mathcal{R}(\eta_t, \delta_t)$. For the special case $k = 1$ we have $\V_1^{(\alpha)} = L_t^{(\alpha)}$, which is the length power functional, whose covariance matrix was already found in \cite{Reitzner} and algebraically investigated in \cite{ReitznerRoemer}. It sums up the $\alpha$-th powers of the edge lengths in the random geometric graph $\mathcal{G}( \eta_t, \, \delta_t)$ which is the underlying graph of $\mathcal{R}(\eta_t, \delta_t)$.
Furthermore, observe that for the power $\alpha = 0$ the functional $\V_k^{(0)}$ coincides with the number of $k$-simplices in $\mathcal{R}( \eta_t, \, \delta_t)$. In particular, the vector of volume power functionals $(\V_0^{(0)}, \ldots, \V_{\text{dim}(\mathcal{R}(\eta_t, \delta_t))-1}^{(0)})$ coincides with the $f$-vector of $\mathcal{R}(\eta_t, \delta_t)$.
See \cite{AkinwandeReitzner, Grygierek} and \cite{ReitznerRoemer} for a stochastic discussion and special cases in the context of the volume power functional.

To prepare the studies in the next sections, we record some important and helpful results. First, we notice, that the expectation of the volume power functional was already determined in \cite[Theorem 3.1]{AkinwandeReitzner}.
\begin{thm}
    Let $\alpha > -d+k-1$ for $k \leq d$ and $\alpha  = 0 $ for $k >d$. The expectation of the volume power functional is given by
    \[\mathbb{E} \mathcal{V}_k^{(\alpha)} = \frac{\mu_k^{(\alpha)}}{(k+1)!}t^{k+1}\delta_t^{k(\alpha+d)} (1+O(\delta_t)).\]
\end{thm}

The algebraic examinations in this manuscript are based on the covariance matrix from \cite{AkinwandeReitzner} which was found for a modified vector of volume power functionals 
\[(\tilde{\V}_{k_1}^{(\alpha_1)}, \ldots ,\tilde{\V}_{k_n}^{(\alpha_n)}).\]
The normalized volume power functional used in the following sections is defined through
\[\tilde{\V}_{k_i}^{(\alpha_i)}= \frac{\V_{k_i}^{(\alpha_i)}}{Q_i},\] 
with 
\begin{align}
\label{Eq:Q_i}
    Q_i = t^{\frac{1}{2}}\delta_t^{\alpha_ik_i}\max_{1\leq m \leq k_i+1}\{(t\delta_t^d)^{k_i-\frac{m-1}{2}}\}.
\end{align}
 Here, the sequence of pairs $(k_i, \alpha_i)$ needs to be an \textit{admissible sequence}. This means, that it satisfies the following conditions:
\begin{enumerate}
    \item $0 \leq k_1 \leq \ldots \leq k_n$,
    \item the pairs are distinct, 
    \item $\alpha_i> -d+k_i-1$ and $\alpha_i +\alpha_j >-d+\min_{l \in \{i,j\}}k_l-1$ for all $i,j \in [n]$,
    \item $\alpha_i=0$ if $k_i>d$.
\end{enumerate}

The covariance matrix of a functional vector $(\tilde{\V}_{k_1}^{(\alpha_1)}, \ldots ,\tilde{\V}_{k_n}^{(\alpha_n)}) $ of an admissible sequence was explicitly determined by Akinwande and Reitzner \cite{AkinwandeReitzner} for all regimes. Note however, that the critical regime is divided into two further cases.

\begin{thm} [(Akinwande--Reitzner)]
\label{Satz: Kov.matrix Laengen Potenz Funktional}
Let $(k_1, \alpha_1), \ldots ,(k_n, \alpha_n)$ be an admissible sequence. The random vector $(\tilde{\V}_{k_1}^{(\alpha_1)}, \ldots, \tilde{\V}_{k_n}^{(\alpha_n)}) $ has the asymptotic covariance matrix 
\begin{align} 
\Sigma_n := \begin{cases}  
A_0^{<1} &\text{ for } \, \lim\limits_{t \rightarrow \infty} t\delta_t^d =0,  \\
\sum_{m=0}^{2k_n}A_m^{<1}c^{m/2}& \text{ for } \,\lim\limits_{t \rightarrow \infty} t\delta_t^d =c \in (0, \, 1],   \\
\sum_{m=0}^{k_n}A_m^{>1}c^{-m} & \text{ for }\,\lim\limits_{t \rightarrow \infty} t\delta_t^d \in [1, \,\infty), \\
A_0^{>1} &\text{ for }  \,\lim\limits_{t \rightarrow \infty} t\delta_t^d =\infty \\
\end{cases} \label{Eq: Kov.matrix Laengen Potenz Funktional} 
\end{align} 
with 
\begin{align} \label{Eq: Am<1}
A_m^{>1} := \left( \mu_{k_l,k_j:m+1}^{(\alpha_l,\alpha_j)} \frac{\mathds{1}(m\leq \min_{i\in \{l,j\}} k_i)}{(m+1)!(k_l-m)!(k_j-m)!}\right)_{l,j \in [n]}    
\end{align}
 and 
\begin{align}
\label{Eq: Am>1}
  A_m^{<1}  := \left(\mu_{k_l,k_j: \frac{k_l+k_j-m+2}{2}}^{(\alpha_l,\alpha_j)} \frac{\mathds{1}(m-|k_l-k_j| \in \{0,2,4, \ldots, 2 \min_{i\in \{l,j\}} k_i\})}{(\frac{k_l+k_j-m+2}{2})!(\frac{m-k_l+k_j}{2})!(\frac{m+k_l-k_j}{2})!}\right) _{l,j \in [n]}.  
\end{align}

\end{thm}

\section{Main results} \label{Kap: Main results}

The following tables provide an overview on the algebraic questions studied in this work concerning the asymptotic covariance matrices $\Sigma_n$ of a vector of volume power functionals from (\ref{Eq: Kov.matrix Laengen Potenz Funktional}) for $n > 1$. They summarize all the results we obtained in each regime and highlight open problems. These tables show, whether a result has been found in the respective regime for the corresponding vector. $\surd$ indicates that a result (e.g., a formula) has been derived for the considered property. $\times$ signifies that the issue remains an open problem, while $o$ denotes that results have been found for special situations or cases, or that bounds have been established for the exact values, which have yet to be determined. For $n = 1$, all aspects are trivial.

\begin{center} \label{table: first results}
\captionsetup{type=table}
\begin{tabular}[h]{| F x{1pt}c|c|c|E|E|E|}

\toprule    
& \textbf{supercritical regime } & \textbf{subcritical regime } & \textbf{critical regime } \\

\midrule
\textbf{expectation}&$\surd$&$\surd$& $\surd$\\
\midrule
\textbf{rank}&$\surd$&  $\surd$& $\surd$\\
\midrule

\textbf{Jordan}& $\surd$&o&$\times$ \\

\textbf{normal form}&& &\\
\midrule 
 
\textbf{determinant} &  $\surd$& o& $\times$ \\ 
\midrule

\textbf{definiteness} & $\surd$& $\surd$& $\surd$ \\ 
\midrule

 \textbf{inverse matrix}& $\times$&o& $\times$ \\
\midrule

\textbf{eigenvalues}& $\surd$& o& o\\
\midrule    

\textbf{eigenspaces}& $\surd$ & o&$\times$\\
\bottomrule

\end{tabular}
\caption{\small{First results in all regimes}} \label{Tab: first results in all regimes}
\end{center}

\begin{center} \label{table: decompositions}
\captionsetup{type=table}
\begin{tabular}[h]{| F x{1pt}c|c|c|E|E|E|}

\toprule    

\midrule
& \textbf{supercritical regime } & \textbf{subcritical regime } & \textbf{critical regime } \\
\toprule
\textbf{LU}&$\surd$& o& $\times$\\
\midrule
 
\textbf{Cholesky}& $\surd$&o& $\times$  \\
\midrule
 
\textbf{matrix root}& $\surd$&o&$\times$ \\

\bottomrule

\end{tabular}
\caption{\small{Decompositions in all regimes}} \label{Tab: decompositions in all regimes}
\end{center}

 \section{Applications} \label{Sec: stoch. appl.}
In this section we present use-cases of the algebraic properties from Section \ref{Kap: Main results}. Certain scenarios can be addressed by applying results from Tables \ref{Tab: first results in all regimes} and \ref{Tab: decompositions in all regimes}.
First, the $f$-functionals $F_k$ which count the number of $k$-dimensional simplices in a simplicial complex are considered. Recall, that they are special cases of volume power functionals for $\alpha=0$ and corresponding dimension $k$. Consequently, the results presented in the tables in Section \ref{Kap: Main results} are applicable to these cases as well. We will use them to derive findings for the truncated $h$-vector of $\mathcal{R}(\eta_t, \delta_t)$, which consists of the functionals
\begin{align} \label{h-functional}
H_k:&= h_k (\mathcal{R}(\eta_t, \delta_t)) = (-1)^k\binom{l}{k}+\sum_{i=1}^k (-1) ^{k-i}\binom{l-i}{k-i} \frac{1}{i!} \sum _{ (y_0, \ldots, y_{i-1}) \in \eta_{t,\neq}^{i}} \mathds{1}_{\leq \delta_t}(y_0, \ldots, y_{i-1})\\
&= (-1)^k\binom{l}{k}+\sum_{i=1}^k (-1) ^{k-i}\binom{l-i}{k-i}F_{i-1}  = (-1)^k\binom{l}{k}+\sum_{i=1}^k (-1) ^{k-i}\binom{l-i}{k-i}\mathcal{V}_{i-1}^{(0)},\nonumber
\end{align}
where \begin{align}
    \label{Eq: l}l= \min(d, \dim(\mathcal{R}(\eta_t, \delta_t))
\end{align} and $k\in [l]$ as well as $H_0 =1$. In the case $d < \dim(\mathcal{R}(\eta_t, \delta_t))$ we only consider the induced simplicial sub complex $\mathcal{R}_l(\eta_t, \delta_t)$ of $\mathcal{R}(\eta_t, \delta_t)$ which fulfills \[\dim(\mathcal{R}_l(\eta_t, \delta_t)) = d,\] where $d$ is as before the dimension of the surrounding space. 

The authors from \cite{ReitznerRoemer} derived stochastic applications for length power functionals. In the following paragraphs some of their strategies are adapted. 
This enables us to derive immediately stochastic conclusions about the  volume power functional. The proofs follow along the proofs of \cite[Section 4]{ReitznerRoemer}.

First, we show that the eigendecomposition $\Sigma_n= S D S^{-1}$ of covariance matrices is of high relevance. Consider 
$$
\mathbb{V} (b_1 \tilde{\mathcal{V}}_{k_1}^{(\alpha_1)}+ \dots + b_n \tilde{\mathcal{V}}_{k_n}^{(\alpha_n)} ) 
=
\mathbb{V} ((b^t S) (S^{-1} (\tilde{\mathcal{V}}_{k_i}^{(\alpha_i)})_{i=1, \dots n}))=
(S^{-1} b)  D (b^t S).
$$
Since $\Sigma_n$ is symmetric, it admits a Schur decomposition with the identical diagonal matrix $D$, whose entries are the eigenvalues $\lambda_1 \leq \dots \leq \lambda_n$ of $\Sigma_n$. In this case, $S$ is an orthogonal matrix of eigenvectors, implying $S^{-1} = S^t$ (i.e., $S \in O(n)$).

We collect results for $\lambda_1 \leq \dots \leq \lambda_n$ in the supercritical regime (see Theorem \ref{Thm: EW.e, ER.e superkrit. Reg.}), as well as under some conditions in the subcritical regime (see Theorems \ref{Thm: eigenvalues subcritical} and \ref{Thm: eigenvalues distinct k_i}). In the critical regime we derive an upper bound for the eigenvalues under certain assumptions (see Conjecture \ref{conj: bounds eigenvalues crit reg}). 
These findings yield the following variance bounds for the volume power functional and in particular for the $f$-functional.

\begin{cor} \label{cor: bounds variance}
Let $\Sigma_n$ be as in Theorem \ref{Satz: Kov.matrix Laengen Potenz Funktional}. Furthermore, let $ b = (b_1, \ldots, b_n) \in \mathbb{R}^n$ and $\underline S_n, \overline S_n$ be bounds for the eigenvalues of $\Sigma_n$. Then
\begin{align} \label{Eq: variance bounds}
   \|b \|^2 \underline S_n \leq \| b \|^2 \lambda_1 \leq
\mathbb{V}  (b_1 \tilde{\mathcal{V}}_{k_1}^{(\alpha_1)}+ \dots + b_n \tilde{\mathcal{V}}_{k_n}^{(\alpha_n)} ) 
\leq \| b \|^2 \lambda_n \leq 
\| b \|^2 \overline S_n.  
\end{align}

\end{cor}

The results for the eigenvalues and their bounds mentioned above in combination with (\ref{Eq: variance bounds}) offer explicit formulas for the variance. According to these findings one gains some important information on $h$-functionals which are indeed linear combinations of $f$-functionals.

That means that choosing the vector $(b_1, \ldots, b_n)$ along (\ref{h-functional}) gives bounds for the variance of $h$-functionals. The observations so far lead to the following asymptotic conclusion.
\begin{lem} \label{thm: h-functional}
Let $j \in [l]$, where $l$ is defined according to (\ref{Eq: l}). Moreover, let $\lambda_1 \leq \ldots \leq \lambda_{j}$ be the eigenvalues of $\Sigma_{j}$ from (\ref{Eq: Kov.matrix Laengen Potenz Funktional}) associated with the volume power vector $(\tilde{\mathcal{V}}_{0}^{(0)}, \ldots , \tilde{\mathcal{V}}_{j-1}^{(0)})$ and let $\underline S_j, \overline S_j $ be bounds for the eigenvalues of $\Sigma_j$.
The $h$-functional $H_j$ satisfies
\begin{align} \label{Eq: h-functional}
    \sum_{i=1}^j (Q_i(-1)^{j-i}\binom{l-i}{j-i})^2 \underline S_j &\leq \sum_{i=1}^j (Q_i(-1)^{j-i}\binom{l-i}{j-i})^2 \lambda_1 \\
    &\leq \mathbb{V}  (H_j )\leq  \nonumber\\ 
&\sum_{i=1}^j (Q_i(-1)^{j-i}\binom{l-i}{j-i})^2\lambda_j \leq \sum_{i=1}^j(Q_i(-1)^{j-i}\binom{l-i}{j-i})^2\overline S_j \nonumber,
\end{align}
 where $Q_i$ is according to (\ref{Eq:Q_i}).

\end{lem}

\begin{proof}
To obtain a statement about the $h$-functional $H_j$ from (\ref{h-functional}), choose the $k_1,\ldots, k_j$ from Corollary \ref{cor: bounds variance} as $0, \ldots, j-1$ as well as $\alpha_i = 0$ and 
\[b_i = Q_i(-1)^{j-i} \binom{l-i}{j-i}\] 
for $i \in [j]$. This leads to the norm \[\| b \|=\sqrt {\sum_{i=1}^j (Q_i(-1)^{j-i}\binom{l-i}{j-i})^2}.\] Then we see, that
\begin{align*}
    \|b \|^2 \underline S_j \leq \| b \|^2 \lambda_1 &\leq \mathbb{V}  (b_1 \tilde{\mathcal{V}}_{0}^{(0)}+ \dots + b_j \tilde{\mathcal{V}}_{j-1}^{(0)} )\\
    &=
\mathbb{V}  ((-1)^j\binom{l}{j}+b_1 \tilde{\mathcal{V}}_{0}^{(0)}+ \dots + b_j \tilde{\mathcal{V}}_{j-1}^{(0)} ) \\
&= \mathbb{V}  (H_j)   
\leq \| b \|^2 \lambda_j \leq 
\| b \|^2 \overline S_j.
\end{align*}
 Here, the $ \underline S_j, \overline S_j$ are lower and upper bounds for the eigenvalues of the covariance matrix $\Sigma_j$.

\end{proof}
\begin{thm}
For $i \in [k]$ and $m \in [k-1]$ let $a_i = \frac{{(i-1)}!}{\mu_{i-1}^{(0)}}$ and $a_{i,m} = \sqrt{(m+1)!}(i-1-m)!$. The $h$-functional $H_k$ for $k \in [l]$, where $l$ is defined according to (\ref{Eq: l}), satisfies the following statements under the conditions of Lemma \ref{thm: h-functional}: 
\begin{enumerate}
    \item In the supercritical regime, we have $$
 0 \leq
\mathbb{V} (H_k) 
\leq \| b \|^2 \sum_{i=1}^k \frac{1}{a_i^2}.
$$
\item In the subcritical regime, we have \[
 \| b \|^2 \mu_{0}^{(0)}\leq
\mathbb{V} (H_k) 
\leq \| b \|^2 \frac{\mu_{k-1}^{(0)}}{ k!}.
\]
\item In the critical regime, for $c \in [1, \infty)$, assuming Requirement \ref{Asum 1}, we have$$
 0 \leq
\mathbb{V} (H_k) 
\leq k\| b \|^2 \cdot \max_{m=0, \ldots, k-1}\left(\kappa_d^{2k-m-2}(\sum_{i=1}^{k} \frac{1}{ a_{i,m}^2})\right).
$$ 
\end{enumerate}
\end{thm}

\begin{proof} Lemma \ref{thm: h-functional} leads to explicit results for the variance of $H_k$ with $k \in [l]$ by considering in each regime the vector $(\tilde{\mathcal{V}}^{(0)}_{k_1}, \ldots, \tilde{\mathcal{V}}^{(0)}_{k_n})=(\tilde{\mathcal{V}}^{(0)}_0, \ldots, \tilde{\mathcal{V}}^{(0)}_{k-1})$ and the corresponding covariance matrix $\Sigma_k$. \\
 In the supercritical regime, we gain insights into the eigenvalues \[ \lambda_1, \ldots, \lambda_{k-1} =0 \quad\text{ and } \quad\lambda_k =\sum_{i=1}^k 1/ a_i^2\] of $\Sigma_k$ in Theorem \ref{Thm: EW.e, ER.e superkrit. Reg.}, enabling their application to Equation (\ref{Eq: h-functional}). This proves (i). \\
For the vector $(\tilde{F}_0, \ldots, \tilde{F}_{k-1})$, we consider a special case of the vector of distinct $k_i$ from Definition \ref{Def: vector distinct ki}. Consequently, we also provide results for the eigenvalues of $\Sigma_k$ in the subcritical regime in Theorem \ref{Thm: eigenvalues distinct k_i}. Here, $\lambda_i = \frac{\mu_{i-1}^{(0)}}{ i !}$ for $i \in [k]$, which together with Equation (\ref{Eq: h-functional}) leads directly to the result in (ii). \\
In the critical regime, assuming Requirement \ref{Asum 1} and using Conjecture \ref{conj: bounds eigenvalues crit reg}  we obtain with the simplification $S_m=\kappa_d^{2k-m-2}$ from (\ref{Eq: bound S}) some bounds for the eigenvalues of $\Sigma_k$. Plugging these into Equation (\ref{Eq: h-functional}) yields the findings from (iii).
\end{proof}

Furthermore, there are applications of the inverse matrix, which is called concentration matrix, in the case that $\Sigma_n$ is invertible. The entries of the joint concentration matrix of a volume power vector of length 3 present partial correlations between two functionals in the vector given a third one. See the following theorem and the upcoming example for more details.

\begin{thm}
    Let $(\tilde{\mathcal{V}}^{(\alpha_1)}_{k_1},  \tilde{\mathcal{V}}^{(\alpha_2)}_{k_2}, \tilde{\mathcal{V}}^{(\alpha_3)}_{k_3})$ be a normalized vector of volume power functionals with covariance matrix $\Sigma_3= (\sigma_{ij})$ according to (\ref{Eq: Kov.matrix Laengen Potenz Funktional}). Let $\Sigma_3$ be invertible.
    The partial correlation of $\tilde{\mathcal{V}}^{(\alpha_1)}_{k_1}$ and $\tilde{\mathcal{V}}^{(\alpha_2)}_{k_2}$ given $\tilde{\mathcal{V}}^{(\alpha_3)}_{k_3}$ is
    \[\rho_{\tilde{\mathcal{V}}^{(\alpha_1)}_{k_1}, \tilde{\mathcal{V}}^{(\alpha_2)}_{k_2}| \tilde{\mathcal{V}}^{(\alpha_3)}_{k_3}} = -\frac{s_{12}}{\sqrt{s_{11}s_{22}}},\] where $ \Sigma_3^{-1}= (s_{ij})$ is the inverse matrix of $\Sigma_3$.
\end{thm}

\begin{proof}
    The partial corelation $\rho_{\tilde{\mathcal{V}}^{(\alpha_1)}_{k_1}, \tilde{\mathcal{V}}^{(\alpha_2)}_{k_2}| \tilde{\mathcal{V}}^{(\alpha_3)}_{k_3}}$ is the correlation between the residuals $L_{1}$ and $L_2$ resulting from the linear regression of $\tilde{\mathcal{V}}^{(\alpha_1)}_{k_1}$ with $\tilde{\mathcal{V}}^{(\alpha_3)}_{k_3}$ and of $\tilde{\mathcal{V}}^{(\alpha_2)}_{k_2}$ with $\tilde{\mathcal{V}}^{(\alpha_3)}_{k_3}$. \\
    Thus, we consider \[x_0= \text{argmin}_x\mathbb{E}\| \tilde{\mathcal{V}}^{(\alpha_1)}_{k_1}- x \tilde{\mathcal{V}}^{(\alpha_3)}_{k_3}  \|^2 \text{ and } y_0= \text{argmin}_y\mathbb{E}\| \tilde{\mathcal{V}}^{(\alpha_2)}_{k_2}- y \tilde{\mathcal{V}}^{(\alpha_3)}_{k_3}  \|^2 \] 
    and obtain
    \[x_0 = \sigma_{33}^{-1}\sigma_{31} \text{ and } y_0= \sigma_{33}^{-1}\sigma_{32}.\]
    This means that one can write the residuals as 
    \begin{align*}
      L_1 &=\tilde{\mathcal{V}}^{(\alpha_1)}_{k_1}-x_0\tilde{\mathcal{V}}^{(\alpha_3)}_{k_3} =  \tilde{\mathcal{V}}^{(\alpha_1)}_{k_1}-\sigma_{33}^{-1}\sigma_{31}\tilde{\mathcal{V}}^{(\alpha_3)}_{k_3}\\
      &\text{ and } \\
      L_2 &=\tilde{\mathcal{V}}^{(\alpha_2)}_{k_2}-y_0\tilde{\mathcal{V}}^{(\alpha_3)}_{k_3} = \tilde{\mathcal{V}}^{(\alpha_2)}_{k_2}-\sigma_{33}^{-1}\sigma_{32}\tilde{\mathcal{V}}^{(\alpha_3)}_{k_3}.  
    \end{align*}
 The joint covariance of the residuals is then given by
    \begin{align*}
    \mathbb{C}ov (L_1,L_2) =&\,\mathbb{C}ov(\tilde{\mathcal{V}}^{(\alpha_1)}_{k_1}-\sigma_{33}^{-1}\sigma_{31}\tilde{\mathcal{V}}^{(\alpha_3)}_{k_3},\tilde{\mathcal{V}}^{(\alpha_2)}_{k_2}-\sigma_{33}^{-1}\sigma_{32}\tilde{\mathcal{V}}^{(\alpha_3)}_{k_3})\\
=&\,\mathbb{C}ov(\tilde{\mathcal{V}}^{(\alpha_1)}_{k_1},\tilde{\mathcal{V}}^{(\alpha_2)}_{k_2})- \sigma_{33}^{-1}\sigma_{23}\mathbb{C}ov(\tilde{\mathcal{V}}^{(\alpha_1)}_{k_1},\tilde{\mathcal{V}}^{(\alpha_3)}_{k_3}) - \sigma_{33}^{-1}\sigma_{13}\mathbb{C}ov(\tilde{\mathcal{V}}^{(\alpha_2)}_{k_2},\tilde{\mathcal{V}}^{(\alpha_3)}_{k_3} )\\
&+\sigma_{33}^{-1}\sigma_{33}^{-1}\sigma_{32}\sigma_{31}\mathbb{C}ov(\tilde{\mathcal{V}}^{(\alpha_3)}_{k_3},\tilde{\mathcal{V}}^{(\alpha_3)}_{k_3}) \\
        =& \,\sigma_{12}-\sigma_{33}^{-1}\sigma_{23}\sigma_{13}-\sigma_{33}^{-1}\sigma_{13}\sigma_{23}+\sigma_{33}^{-1}\sigma_{33}^{-1}\sigma_{32}\sigma_{31}\sigma_{33}=\sigma_{12}-\sigma_{13}\sigma_{33}^{-1}\sigma_{23}.
         \end{align*}

         Similarly, the variances of the residuals are given by
         \[\mathbb{C}ov(L_1,L_1) = \sigma_{11}-\sigma_{13}\sigma_{33}^{-1}\sigma_{13} \,\text{  and  } \,\mathbb{C}ov(L_2,L_2) = \sigma_{22}-\sigma_{23}\sigma_{33}^{-1}\sigma_{23}.  \]
    Then, we use the Schur complement and the inverse formula for block matrices (see \cite{Crabtree}). Here, the following block form is useful to compute the inverse of $S$.
    \[\Sigma_3^{-1} = \left( \begin{array}{rrr}
s_{11} & s_{12} & s_{13} \\
s_{12} & s_{22} & s_{23} \\
s_{13} & s_{23} & s_{33} \\
\end{array}\right) = \left( \begin{array}{rrr}
S & R^t \\
R & s_{33} \\
\end{array}\right), \text{ with } S=\left( \begin{array}{rrr}
s_{11} & s_{12}  \\
s_{12} & s_{22}  \\
\end{array}\right) \text{ and } R = (s_{13}, s_{23}). \] Now consider the decomposition of $\Sigma_3$ into
 \[\Sigma_3 = \left( \begin{array}{rrr}
\sigma_{11} & \sigma_{12} & \sigma_{13} \\
\sigma_{12} & \sigma_{22} & \sigma_{23} \\
\sigma_{13} & \sigma_{23} & \sigma_{33} \\
\end{array}\right) = \left( \begin{array}{rrr}
A & B^t \\
B & \sigma_{33} \\
\end{array}\right), \text{ with } A=\left( \begin{array}{rrr}
\sigma_{11} & \sigma_{12}  \\
\sigma_{12} & \sigma_{22}  \\
\end{array}\right) \text{ and } B = (\sigma_{13}, \sigma_{23}).\]
 This leads to
\[S^{-1} = A-B^t\sigma_{33}^{-1}B= (\sigma_{ij})_{i,j\in \{1,2\}}- (\sigma_{31}, \sigma_{32})^t \sigma_{33} ^{-1}(\sigma_{13},\sigma_{23}).\]
    
    Define $S^{-1}:= (s_{ij}^{(-1)})$ and see that for $i, j \in \{1,2\}$ holds
    \begin{align*}
       s_{ij}^{(-1)}= \sigma_{ij}-\sigma_{3i}\sigma_{33}^{-1}\sigma_{3j} = \mathbb{C}ov(L_i,L_j),
    \end{align*}
    which are exactly the covariances of the residuals above. 
   Finally, apply the determinant formula for $2 \times 2$-matrices to recompute $S^{-1}$ and obtain
    \[S^{-1}=\frac{1}{s_{11}s_{22}-s_{12}^2 }\left( \begin{array}{rrr}
s_{22} & -s_{12}  \\
-s_{12} & s_{11}  \\
\end{array}\right).\]
A comparison yields 
\begin{align*}
    \mathbb{C}ov(L_1,L_1)&=  \frac{s_{22}}{s_{11}s_{22}-s_{12}^2 },  \,\mathbb{C}ov(L_2,L_2)=  \frac{s_{11}}{s_{11}s_{22}-s_{12}^2 }\\
    &\text{ and } \,\,\mathbb{C}ov(L_1,L_2)=  \frac{-s_{21}}{s_{11}s_{22}-s_{12}^2 }.
\end{align*}

  To finalize the proof, conclude

\[ \rho_{\tilde{\mathcal{V}}^{(\alpha_1)}_{k_1}, \tilde{\mathcal{V}}^{(\alpha_2)}_{k_2}| \tilde{\mathcal{V}}^{(\alpha_3)}_{k_3}} = \frac{\mathbb{C}ov (L_1,L_2)}{\sqrt{\mathbb{C}ov (L_1,L_1)\cdot \mathbb{C}ov (L_2,L_2)}}=  
-\frac{s_{12}}{\sqrt{s_{11}s_{22}}}. \qedhere
\]

\end{proof}

\begin{exa}
   In the subcritical regime we can easily compute the concentration matrix for the vector of distinct $k_i$ (see Definition \ref{Def: vector distinct ki}), as it is a diagonal matrix. For a vector of length $n=3$ this yields the matrix

    \[\Sigma_3^{-1}= \left( \begin{array}{rrr}
\sigma_{11}^{-1} & 0 & 0 \\
0 & \sigma_{22}^{-1}  & 0 \\
0 & 0 & \sigma_{33}^{-1}  \\
\end{array}\right)\] and therefore 
$\rho_{\tilde{\mathcal{V}}^{(\alpha_1)}_{k_1}, \tilde{\mathcal{V}}^{(\alpha_2)}_{k_2}| \tilde{\mathcal{V}}^{(\alpha_3)}_{k_3}} = 0.$
\end{exa}

Moreover, we can make further use of the Eigenvalue decomposition in the supercritical regime which is, in dependence of the $a_i$ from (\ref{Eq: Abbr}) the same procedure as for the special case of the length power functional in \cite{ReitznerRoemer}. Therefore, by leveraging these results, we are able to derive the following relation between volume power functionals. According to \cite[Theorem 4.5.]{AkinwandeReitzner} $(\tilde{\mathcal{V}}_{k_1}^{(\alpha_1)}, \ldots ,\tilde{\mathcal{V}}_{k_n}^{(\alpha_n)})$ is asymptotic normal distributed under the given conditions there.

\begin{thm}\label{thm:stochconv-dense}
For $t\delta_t^d \rightarrow \infty$ let $(k_1,\alpha_1), (k_2,\alpha_2)$ be an admissible sequence with $4\alpha_i >-d+k_i-1$ for $i\in \{1,2\}$. Then define
$$
D_t^{(\alpha_1, \alpha_2)} = a_1 \tilde{ \mathcal{V}}_{k_1}^{(\alpha_1)} - a_2 \tilde{ \mathcal{V}}_{k_2}^{(\alpha_2)} \text{ and } \hat{D}_t^{(\alpha_1, \alpha_2)} = D_t^{(\alpha_1, \alpha_2)}- \mathbb{E}(D_t^{(\alpha_1, \alpha_2)})
$$
The following statements holds:
\[\hat{D}_t^{(\alpha_1, \alpha_2)} \stackrel{\mathds P}{\to} 0 
 \text{ as $t \to \infty$ }.\]

\end{thm}
\begin{proof}
We show that the fluctuations of the two functionals are asymptotically identical. Specifically, the centered difference $\hat{D}_t^{(\alpha_1, \alpha_2)}$ converges in probability to $0$, which implies that the two functionals are perfectly correlated in the limit.
By orthogonalizing the matrix $S$ from Corollary \ref{Kor: Diag.matrix, char. Poly. superkrit. Reg.}, the Schur decomposition of $\Sigma_2$ is obtained, yielding the rotation matrix
\[
R = \begin{pmatrix} 
\frac{a_1}{\sqrt{a_1^2 +a_2^2}} & \frac{-a_2}{\sqrt{a_1^2 +a_2^2}}  \\ 
\frac{a_2}{\sqrt{a_1^2 +a_2^2}} & \frac{a_1}{\sqrt{a_1^2 +a_2^2}}\\ 
\end{pmatrix},
\]
such that $R^{-1}(\tilde{\mathcal{V}}_{k_1}^{(\alpha_1)}, \tilde{\mathcal{V}}_{k_2}^{(\alpha_2)}) \xrightarrow{t \to \infty} N$ with $N \sim \mathcal{N}(0, \Sigma_2)$, where
\[
\Sigma_2 = \begin{pmatrix} 
0 & 0 \\ 
0 & \left(d\kappa_d\right)^2 \left( \sum_{i=1}^n \frac{1}{a_i} \right) 
\end{pmatrix}.
\]

The first component of the resulting vector is exactly $\frac{1}{\sqrt{a_1^2+a_2^2}} D_t^{(\alpha_1, \alpha_2)}$. Since the first diagonal entry of $\Sigma_2$ is $0$, it follows that $\mathbb{V}(D_t^{(\alpha_1, \alpha_2)}) \to 0$ as $t \to \infty$. By Chebyshev's inequality, the centered variable $\hat{D}_t^{(\alpha_1, \alpha_2)}$ converges in probability to $0$, which completes the proof.
\end{proof}

The order of $D_t^{(\alpha_1, \alpha_2)}$ and $D_t^{(\alpha)} := D_t^{(\alpha, 0)}$ in Theorem \ref{thm:stochconv-dense}, as well as in the subsequent corollaries and theorems of this section, follows from the quantitative multivariate central limit theorem in \cite[Theorem 4.5]{AkinwandeReitzner}, analogous to the approach used in \cite{ReitznerRoemer}. This means, for $\lim_{t \to \infty} t \delta_t^d = c \in [0, \infty]$ we have
\begin{equation}\label{eq:order-Dt}
\| D_t^{(\alpha)} \|_{d_3} = O(t^{-\frac 12} \max\{1,(t \delta_t^d)^{- \frac 12}\}) + O(\delta_t) 
+ O(\min\{|c-t \delta_t^d|, (t \delta_t^d)^{-1} \}).
\end{equation}

Following the approach in \cite{ReitznerRoemer}, setting $\alpha_1=0$ allows us to recover several special cases. Recall, that $a_i =\frac{k_{i}!}{\mu_{k_{i}}^{(\alpha_i)}}$.  In particular, for $\alpha=0$ holds  
\begin{align} \label{Eq: muk0}
    \mu_k^{(0)}=\kappa_d^k\mathbb{P}(\Delta_1[0,\{a_l\}^k_{l=1}] \neq 0),
\end{align}
 which is the scaled probability that $k$ random points in $B_d$ have pairwise distances at most
one.
The result, that $\hat{D}_t \stackrel{\mathds P}{\to} 0$, is a direct consequence of the fact that $\Sigma_n$ is a rank-one matrix. This implies a total dominance of the first principal component, meaning that all volume power functionals fluctuate in perfect synchrony in the supercritical regime.
In the special case where the dimensions match ($k_1 = k_2 = k$), the expectations cancel out, allowing us to express the relationship directly. 

\begin{cor}\label{cor:dense-conv}
Let $(k,0)$ and $(k,\alpha)$ be an admissible sequence such that $4\alpha >-d+k-1$ in the supercritical regime. The number of $k$-dimensional simplices, described through $\tilde{\mathcal{V}}_{k}^{(0)}$, asymptotically determines $\tilde{\mathcal{V}}_{k}^{(\alpha)}$,
\begin{equation}\label{eq:stochconv-dense}
\tilde{\mathcal{V}}_{k}^{(\alpha)} = \frac {\mu_{k}^{(\alpha)}} {\mu_{k}^{(0)}} \tilde{\mathcal{V}}_{k}^{(0)}+ D_t^{(\alpha)},
\end{equation} 
where $\mu_k^{(0)}$ is as in (\ref{Eq: muk0}) and $ D_t^{(\alpha)} $ satisfies (\ref{eq:order-Dt}).
\end{cor}
Numerous further relations exist between $\tilde{\mathcal{V}}_{k_1}^{(\alpha_1)}, \ldots, \tilde{\mathcal{V}}_{k_n}^{(\alpha_n)}$. Since $\Sigma_n$ is of rank one in the supercritical regime, all of these relations can be deduced from Theorem \ref{thm:stochconv-dense}.
$\Sigma_n$ has a rank-one structure in the supercritical regime, allowing it to be approximated through \[\Sigma_n \approx \lambda_2 (v_n v_n^t),\] where $\lambda_2$ is the dominant eigenvalue (see Theorem \ref{Thm: EW.e, ER.e superkrit. Reg.}). To decouple the general system, we employ a Schur decomposition. Define an orthogonal rotation matrix $R \in \mathbb{R}^{n \times n}$, whose columns are the normalized eigenvectors of $\Sigma_n$. Here, the first $n-1$ columns are the normalized vectors $v_1, \dots, v_{n-1}$ corresponding to $\lambda_1 = 0$, and the last column is the normalized dominant eigenvector $v_n$ corresponding to the eigenvalue $\lambda_2$. The transformation $Y = R^{-1} \cdot (\tilde{\mathcal{V}}_{k_1}^{(\alpha_1)}, \ldots, \tilde{\mathcal{V}}_{k_n}^{(\alpha_n)})$ rotates the coordinate system such that the first $n-1$ components $Y_1, \dots, Y_{n-1}$ converge to zero in probability and the last component $Y_n$ captures nearly all of the system's variance. \\
This procedure reduces the $n$-dimensional stochastic process into a one-dimensional problem. Although we consider $n$ different volume power functionals, they do not provide independent information. Instead, they scale proportionally to the same underlying random variable. Exactly this redundant information is captured by the relative sensitivity $\mathcal{R}_{ij}$, which is defined by
$$\mathcal{R}_{ij} = \frac{(v_n)_i}{(v_n)_j} = \frac{\mu_{k_i}^{(\alpha_i)} k_j!}{\mu_{k_j}^{(\alpha_j)} k_i!}.$$
Since the diagonal elements of $\Sigma_n$ represent the variances of the functionals, $\mathcal{R}_{ij}$ directly measures the relative stochastic contribution of each functional to the total variance.
To validate this structure, we employ a Monte Carlo integration approach for $\mu_k^{(\alpha)}$ with fixed values $k=2$ and $d=10$. To avoid dimensionality collapse, we project the sampled points onto their spanning $k$-dimensional subspace by using Singular Value Decomposition before computing the volume of the convex hull. In the special case where $k_i = k_j$, the sensitivity $\mathcal{R}$ is plotted as a function of $\alpha$ in Figure \ref{fig:sensitivity}, relative to a reference value $\alpha_{ref}$.
$\mathcal{R}$ decreases monotonically, since $\Delta_1 \in [0,1]$ and higher powers $\alpha$ attenuate the expected volumes $\mu_k^{(\alpha)}$, while reducing their relative contribution to the total variance. Additionally, the fact that $\mathcal{R}$ forms a smooth, consistent curve confirms that the rank-one approximation holds across a wide range of $\alpha$, as a full-rank matrix would not allow $\mathcal{R}$ to be a sufficient descriptor of the system's stochasticity.
\begin{figure}[H] 
\centering
\includegraphics[scale=0.35]{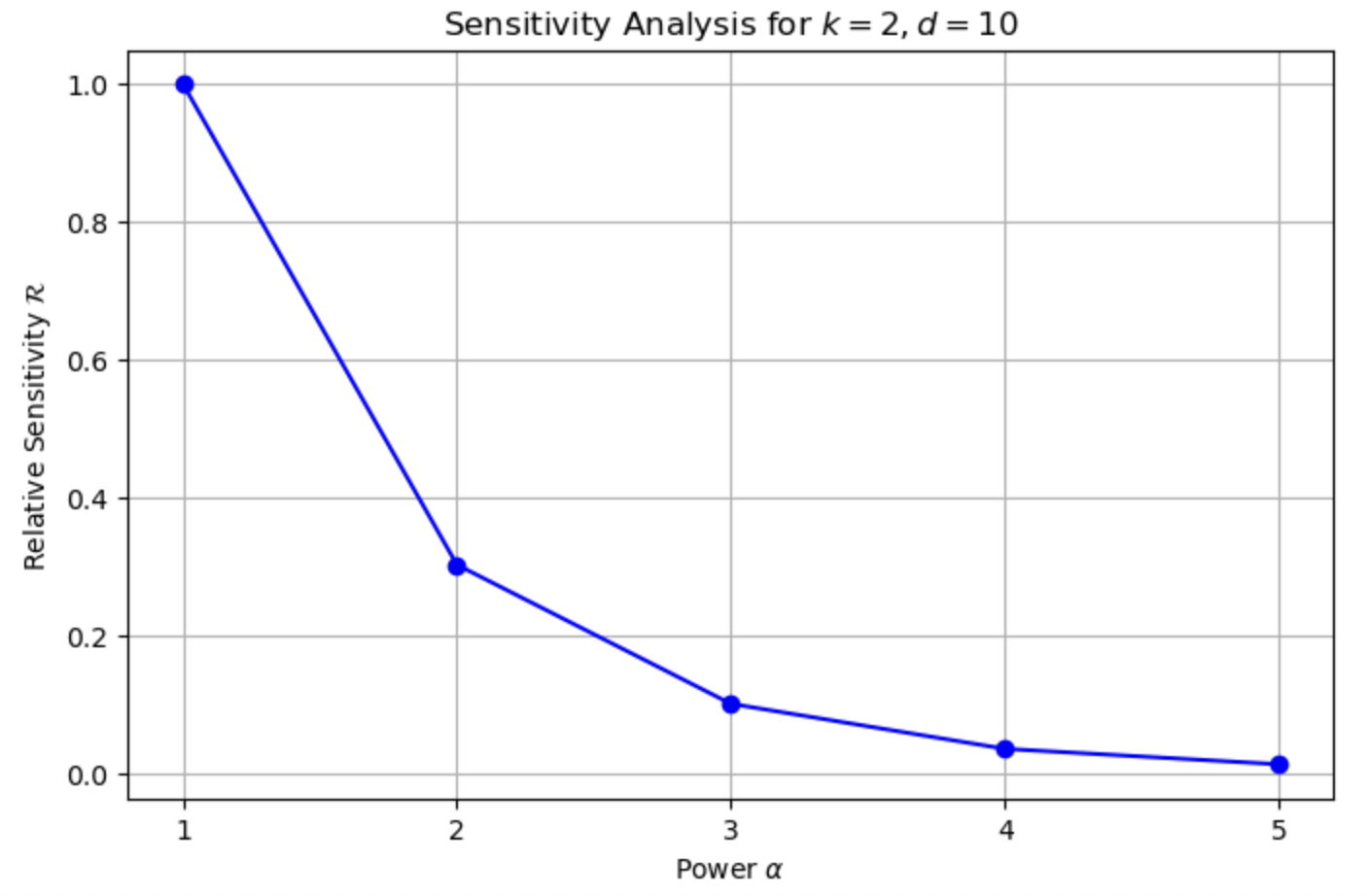}
\caption{Relative sensitivity $\mathcal{R}$ against $\alpha$ in the supercritical regime.}
\label{fig:sensitivity}
\end{figure}

Now, suppose that $Y \sim \mathcal{N}(0, \Sigma_2)$, where $\Sigma_2$ is a positive definite $2 \times 2$-matrix in the subcritical regime. As $\Sigma_2$ may fail to be positive definite in the critical regime, this case is excluded from our current discussion. We study the Cholesky decomposition in a similar way as in \cite{ReitznerRoemer}. For this let $G$ be the Cholesky factor of $\Sigma_2$, i.e., $\Sigma_2 = GG^t$. The identity

$$
G^{-1} \Sigma_2 (G^{-1})^t = I_2
$$

holds, implying that

$$
G^{-1} Y = (Z_1, Z_2),
$$

where $Z_1$ and $Z_2$ are independent $\mathcal{N}(0,1)$-random variables. For $
Y = \lim_{t \to \infty} (\tilde{\mathcal{V}}_{k_1}^{(\alpha_1)}, \tilde{\mathcal{V}}_{k_2}^{(\alpha_2)})
$ we have

$$
G^{-1} Y = \lim_{t \to \infty} \left( g_{11} \tilde{\mathcal{V}}_{k_1}^{(\alpha_1)},\; g_{21} \tilde{\mathcal{V}}_{k_1}^{(\alpha_2)} + g_{22} \tilde{\mathcal{V}}_{k_2}^{(\alpha_2)} \right),
$$
 where $G^{-1}= (g_{ij})$. This leads to the upcoming theorem. But first, apply Theorem \ref{thm: compositions subcritical} and study the special case of distinct $k_i$-simplices. Here, it holds that
\[G = (\Sigma_2)^{1/2},\] as $\Sigma_2$ is diagonal. This implies that
\[ G^{-1} =((\Sigma_2)^{1/2})^{-1}. \]
Since $\Sigma_2=(\sigma_{ij})$ is a $2 \times 2$-matrix, it holds that
\[G =  \left( \begin{array}{cccc}   
\sqrt{\sigma_{11}}& 0 \\
0& \sqrt{\sigma_{22}}\\
\end{array}\right) \] 
and therefore
\[G^{-1} = \left( \begin{array}{cccc}   
g_{11}& 0 \\
0& g_{22}\\
\end{array}\right)= \left( \begin{array}{cccc}   
\frac{1}{\sqrt{\sigma_{11}}}& 0 \\
0& \frac{1}{\sqrt{\sigma_{22}}}\\
\end{array}\right).\] 
There is no row in $G^{-1}$ with non-zero entries in both the first and second columns. Consequently, no relationship between the functionals $\tilde{\mathcal{V}}_{k_1}^{(\alpha_1)}$ and $\tilde{\mathcal{V}}_{k_2}^{(\alpha_2)}$ can be derived from this consideration.

Moreover, it is possible to compute the Cholesky decomposition for $\Sigma_2$ in the subcritical regime with respect to non-distinct $k_i$. In this case, the matrices are
\[G = \left( \begin{array}{cccc}   
\sqrt{\frac{\mu_k^{(2\alpha_1)}}{(k+1)!}}& 0 \\
\frac{\mu_k^{(\alpha_1+\alpha_2)}}{\sqrt{(k+1)!\mu_k^{(2\alpha_1)}}}& \sqrt{\frac{\mu_k^{(2\alpha_2)}\mu_k^{(2\alpha_1)}-(\mu_k^{(\alpha_1+\alpha_2})^2}{(k+1)!\mu_k^{(2\alpha_1)}}}\\
\end{array}\right)\] 
and 
\begin{align} \label{Eq: cholesky non-distinct k}
    G^{-1} = \left( \begin{array}{cccc}   
\frac{(k+1)!}{\sqrt{\mu_k^{(2\alpha_1)}}}& 0 \\
-\frac{\sqrt{(k+1)!}\mu_k^{(\alpha_1+\alpha_2)}}{\sqrt{\mu_k^{(2\alpha_1)}(\mu_k^{(2\alpha_2)}\mu_k^{(2\alpha_1)}-(\mu_k^{(\alpha_1+\alpha_2)})^2)}}& \frac{\sqrt{(k+1)!\mu_k^{(2\alpha_1)}}}{\sqrt{\mu_k^{(2\alpha_1)}\mu_k^{(2\alpha_2)}-(\mu_k^{(\alpha_1+\alpha_2)})^2}}\\
\end{array}\right).
\end{align}

\begin{thm}\label{thm:stochconv-(sub)crit}
Let $t\delta_t^d \rightarrow 0$ and $(k,\alpha_1), (k,\alpha_2)$ be an admissible sequence with $4\alpha_i >-d+k-1$ for $i\in \{1,2\}$ and $\hat{\mathcal{V}}_k^{(\alpha_i)} = \tilde{\mathcal{V}}_k^{(\alpha_i)} - \mathbb{E}[\tilde{\mathcal{V}}_k^{(\alpha_i)}]$. Then 
\begin{equation*}
\hat{ \mathcal{V}}_k^{(\alpha_2)} = \frac{\mu_k^{(\alpha_1+\alpha_2)}}{\mu_k^{(2\alpha_1)}}\hat{ \mathcal{V}}_k^{(\alpha_1)}
+ \frac{\sqrt{\mu_k^{(2\alpha_1)}\mu_k^{(2\alpha_2)}-(\mu_k^{(\alpha_1+\alpha_2)})^2}}{\sqrt{(k+1)!\mu_k^{(2\alpha_1)}}} Z 
+ D_t^{(\alpha_1, \alpha_2)} 
\end{equation*}
as $t \to \infty$, with $Z\sim \mathcal N(0,1)$ independent of $\tilde{\mathcal{V}}_k^{(\alpha_1)}$ and 
$D_t^{(\alpha_1, \alpha_2)}$ satisfies (\ref{eq:order-Dt}).
\end{thm}
\begin{proof}

In the subcritical regime for non-distinct $k_i$, the Cholesky decomposition $\Sigma_2= GG^t$ leads to $G^{-1}$ which is given by (\ref{Eq: cholesky non-distinct k}). By applying this transformation to the centered functionals $\hat{\mathcal{V}}_k^{(\alpha_i)} = \tilde{\mathcal{V}}_k^{(\alpha_i)} - \mathbb{E}[\tilde{\mathcal{V}}_k^{(\alpha_i)}]$, we obtain
$$Z + \hat{D}_t^{(\alpha_1,\alpha_2)} = -\frac{\sqrt{(k+1)!}\mu_k^{(\alpha_1+\alpha_2)}\hat{\mathcal{V}}_k ^{(\alpha_1)} }{\sqrt{\mu_k^{(2\alpha_1)}(\mu_k^{(2\alpha_2)}\mu_k^{(2\alpha_1)}-(\mu_k^{(\alpha_1+\alpha_2)})^2)}} + \frac{\sqrt{(k+1)!\mu_k^{(2\alpha_1)}} \hat{\mathcal{V}}_k^{(\alpha_2)}}{\sqrt{\mu_k^{(2\alpha_1)}\mu_k^{(2\alpha_2)}-(\mu_k^{(\alpha_1+\alpha_2)})^2}}.$$
Solving for $\hat{\mathcal{V}}_k^{(\alpha_2)}$ yields the desired result, thereby completing the proof.

\end{proof}
As previously noted, in the special case of matching dimensions ($k_1 = k_2 = k$), the expectations vanish, allowing the relationship to be expressed directly in terms of the scaled functionals $\tilde{\mathcal{V}}_k^{(\alpha)}$. The specific case $\alpha_1 = 0$ is presented below and follows a similar logic to the supercritical regime described in Corollary \ref{cor:dense-conv}. \\

\begin{cor} \label{Cor: subcrit alpha 0}
Let $(k,0)$ and $(k,\alpha)$ be an admissible sequence such that $4\alpha >-d+k-1$ in the subcritical regime. The number of $k$-dimensional simplices described through $\tilde{\mathcal{V}}_k^{(0)}$ asymptotically determines $\tilde{\mathcal{V}}_k^{(\alpha)}$ up to a Gaussian noise
$$
\tilde{\mathcal{V}}_k^{(\alpha)} = \frac{\mu_k^{(\alpha)}}{\mu_k^{(0)}}\tilde{\mathcal{V}}_k^{(0)}
+ \frac{\sqrt{\mu_k^{(0)}\mu_k^{(2\alpha)}-(\mu_k^{(\alpha_2)})^2}}{\sqrt{(k+1)!\mu_k^{(0)}}} Z 
+ D_t^{(\alpha)}, 
$$
where $Z\sim \mathcal N(0,1)$ is independent of $\tilde{\mathcal{V}}_k^{(0)}$, $\mu_k^{(0)}$ is as in (\ref{Eq: muk0}) and 
$ D_t^{(\alpha)} $ satisfies (\ref{eq:order-Dt}).
\end{cor}

\section{Algebraic analysis} \label{Section: volume power functional}

Now, we examine $\Sigma_n$ algebraically in each of the regimes from the case distinction in (\ref{Eq: Kov.matrix Laengen Potenz Funktional}). To be more precise, we analyze all the regimes separately in terms of algebraic invariants such as the rank, determinant, definiteness, eigenvalues and decompositions.

\subsection{Supercritical regime}

First, $\Sigma_n$ is considered in the supercritical regime (i.e., $t\delta_t^d \rightarrow \infty$) for $n\geq 2$. 
Recall that in this regime, the following matrix is of interest:

\[
 \Sigma_n
=A_0^{>1} = \Bigl( 
\frac{\mu_{k_i, k_j:1}^{(\alpha_i , \alpha_j)} }{k_j!k_i!} 
 \Bigr)\in  \R^{n \times n}.
\]
For further study, we need to introduce some additional abbreviations.

\begin{align}\label{Eq: Abbr}
  a_i =\frac{k_{i}!}{\mu_{k_{i}}^{(\alpha_i)}}\,\text{ for }\,  i \in [n]\,\,\text{ and }\,\,
   b=\sum_{k=1}^n \prod_{l \in [n] \setminus \{k\}}a_l^2.
\end{align}

Using the $a_i$ and the identity $\mu_{k_i, k_j:1}^{(\alpha_i , \alpha_j)} = \mu_{k_{i}}^{(\alpha_i)} \mu_{k_{j}}^{(\alpha_j)} $, the matrix $\Sigma_n$ can be rewritten as
\begin{equation}
\label{Eq: Sigma superkri}
\Sigma_n
=
\left( 
\begin{array}{cccc}
\frac{1}{a_1^2} & \frac{1}{a_1 a_2}  & \ldots & \frac{1}{a_1 a_n} \\
\frac{1}{a_1a_2} & \frac{1}{a_2^2} & \ldots  & \frac{1}{a_2 a_n}  \\
\vdots & \vdots& \ldots & \vdots \\
\frac{1}{a_1 a_n} & \frac{1}{a_2a_n} & \ldots & \frac{1}{a_n^2}\\
\end{array}
\right).
\end{equation}

Observe, that exactly the matrix structure from (\ref{Eq: Sigma superkri}) appears already in \cite[Equation (9)]{ReitznerRoemer} and has been studied there extensively for the length power vector, which is the special case $k_l= 1$ for $l= 1,\ldots, n$, i.e., $a_l = \alpha_l+d$.
The results from this case can be generalized to the following theorems by replacing the specific values of $a_i$ from \cite{ReitznerRoemer} with the general definition of $a_i$ given in (\ref{Eq: Abbr}) throughout all statements and their respective proofs.

\begin{thm} \label{Thm: EW.e, ER.e superkrit. Reg.} 
Let $\Sigma_n$ be defined as in (\ref{Eq: Kov.matrix Laengen Potenz Funktional}) in the supercritical regime for $n\geq2$. The eigenvalues of $\Sigma_n$ are $\lambda_1 =0$ and $\lambda_2 =\sum_{i=1}^n 1/ a_i^2$.
Let
\[
v_1 = ( \frac{a_1}{a_2},-1,0,\ldots,0)^t,   \ldots ,\, 
v_{n-1}=( \frac{a_1}{a_n},0,\ldots,0,-1,)^t,
\text{ and }
v_n = \left(a_n/a_1 ,\, a_n/a_2 ,\ldots, \, 1\right)^t.
\]
The eigenspaces of $\Sigma_n$ are given through
$\emph{eig}(\Sigma_n, \, \lambda_1)
= 
\langle 
v_1, \ldots, v_{n-1} 
\rangle \text{ and }
\emph{eig}(\Sigma_n, \,\lambda_2) =
\langle 
v_n
\rangle. 
$
\end{thm}

\begin{cor} 
\label{Kor: Rang superkrit. Reg.}
Assuming the situation of Theorem \ref{Thm: EW.e, ER.e superkrit. Reg.}, we can conclude:~(i) $\emph{rank}(\Sigma_n) = 1$; (ii) $\emph{det}(\Sigma_n) =0$;
(iii) $\Sigma_n$ is positive semidefinite, but not positive definite.

\end{cor}

The eigenvalue and eigenspace structure of $\Sigma_n$, as established in the preceding theorem, yields the first matrix decomposition.

\begin{cor} 
\label{Kor: Diag.matrix, char. Poly. superkrit. Reg.}
Assuming the situation of Theorem \ref{Thm: EW.e, ER.e superkrit. Reg.} set  $D=\text{diag}(\lambda_1,\dots,\lambda_1,\lambda_2)\in \R^{n \times n}$
and $S \in \R^{n \times n}$ with column vectors $v_1, \dots,v_n$.
Under the given conditions $\Sigma_n$ is similar to $D$ and $D=S^{-1} \Sigma_n S$, where $S^{-1} = (\tilde{s}_{ij})\in \R^{n \times n}$ with entries
\begin{align*} 
\tilde{s}_{ij} = \left\{
\begin{array}{ll}
\frac{\prod_{k=1}^n a_k^2}{a_{j}a_n b}&\text{ for }\,i=n,\\
-\frac{\sum_{k \in [n] \setminus \{i\}} a_{i}^2\prod_{l\in [n] \setminus \{k, \, i\}} a_{l}^2}{b}&\text{ for}\,j=i+1,\\
\frac{\prod_{k=1}^n a_k^2}{a_{i+1}a_j b}& \text{ else.}\\
\end{array}
\right.
\end{align*} 
\end{cor} 

In the preceding result, we obtained the \textit{Eigenvalue decomposition} of $\Sigma_n$, given by $\Sigma_n = SDS^{-1}$. By orthogonalizing, this leads to a Schur decomposition (see, e.g., \cite{Horn}). In addition to the previous findings, \cite{ReitznerRoemer} also allows for the derivation of LU, Cholesky, and matrix root decompositions.  See the next theorem for details.

\begin{thm}  
\label{Satz: LR-Zerleg. superkrit. Reg.}
For $n\geq2$ let $\Sigma_n$ be defined as in (\ref{Eq: Kov.matrix Laengen Potenz Funktional}) in the supercritical regime. Then the following statements hold.
\begin{enumerate}
\item 
An LU decomposition $\Sigma_n =LU$ of $\Sigma_n$ is given by the matrices $L= (l_{ij})\in \R^{n \times n}$ and $U=(u_{ij})\in \R^{n \times n}$ which have the entries
\begin{align*} 
l_{ij} = \left\{
\begin{array}{ll}
1&\text{ for }\,i=j,\\
\frac{a_1}{a_i} & \text{ for }\,i >1  \text{ and }  j=1,\\
0& \text{ else}.\\
\end{array}
\right.\, \text{ and }\,
u_{ij} = \left\{
\begin{array}{ll}
\frac{1}{a_1a_j} & \text{ for }\,i=1,\\
0& \text{ else}.\\
\end{array}
\right.
\end{align*}
\item
A Cholesky decomposition $ \Sigma_n = G G^t$ of $\Sigma_n$ is given by the matrix $G= (g_{ij})\in \R^{n \times n}$ which has entries 
\begin{align*} 
g_{ij} = \left\{
\begin{array}{ll}
\frac{1}{a_i}& \text{ for } j=1,\\
0& \text{ else}.\\
\end{array} \right.
\end{align*}
\item
A matrix root $B=(b_{ij})\in \R^{n \times n}$ of $\Sigma_n$ 
(i.e., $B^2=\Sigma_n$) has entries
\begin{align*} 
b_{ij} = \frac{\prod_{k=1}^na_k}{a_ia_j \sqrt{\sum_{i=1}^n\prod_{k \in [n] \setminus \{i\}}a_k^2}}.
\end{align*}
\end{enumerate}
\end{thm}

In the upcoming example we illustrate the previous mentioned decompositions of the covariance matrix $\Sigma_3$ of a vector $ (\tilde{\V}_{k_1}^{(\alpha_1)},\tilde{\V}_{k_2}^{(\alpha_2)},\tilde{\V}_{k_3}^{(\alpha_3)})$.

\begin{exa} \label{Bsp: Zerlegungen superkrit. Reg.}

Let $\Sigma_3$ be the covariance matrix according to (\ref{Eq: Kov.matrix Laengen Potenz Funktional}) in the supercritical regime for the volume power vector with the parameters $k_1 = 1, k_2 = 2, k_3 = 3$ and $\alpha_i=1$ for $i \in \{1,2,3\}$ as well as the dimension $d =3$.
The abbreviation $\mu_{k_{i}}^{(\alpha_i)}$ as $\mu_{k_{i}}$ leads to $a_1 = 1/\mu_1, a_2 =2/\mu_2$ and $a_3 = 6 / \mu_3 $. Thus, $\Sigma_3$ is given by
\[\Sigma_3 =
\left( \begin{array}{ccccc}
\mu_1^2& \mu_1 \mu_2/ 2 & \mu_1 \mu_3/ 6   \\
\mu_1 \mu_2/ 2  & \mu_2^2/4 & \mu_2 \mu_3/ 12 \\
\mu_1 \mu_3/ 6   & \mu_2 \mu_3/ 12  & \mu_3^2/36\\
\end{array}\right) .\]
According to Theorem \ref{Satz: LR-Zerleg. superkrit. Reg.}(i) an LU decomposition of $\Sigma_3$ is 
\[  LU = 
\left( \begin{array}{ccccc}
1& 0  & 0 \\
\mu_2/2 \mu_1 & 1 & 0 \\
\mu_3/6 \mu_1 & 0& 1\\
\end{array}\right) \left( \begin{array}{ccccc}
\mu_1^2& \mu_1\mu_2/2  & \mu_1\mu_3/6 \\
0 & 0 & 0 \\
0 & 0& 0 \\
\end{array}\right).\] 
Theorem \ref{Satz: LR-Zerleg. superkrit. Reg.}(ii) leads to the Cholesky decomposition $\Sigma_3= GG^t$, where
\[  GG^t=   
\left( \begin{array}{ccccc}
\mu_1& 0  & 0\\
\mu_2/2 & 0 & 0\\
\mu_3/6& 0 & 0\end{array}\right) \left( \begin{array}{ccccc}
\mu_1&\mu_2/2  & \mu_3/6\\
0 & 0 & 0 \\
0 & 0& 0\\
\end{array}\right).
\]
Theorem \ref{Satz: LR-Zerleg. superkrit. Reg.}(iii) yields $\Sigma_3 =  BB$ as a product of matrix roots. Here, we abbreviate $\prod_{i \in I}\mu_i = \mu_{I}$ and $\{i,j,k\} = i,j,k$ as well as $\{i,j\} = i,j$ and obtain

\begin{align*} B=
&\left( \begin{array}{ccccc}
\frac{\sqrt{72}\mu_1 ^2}{\sqrt{\mu_{1,2,3}(\mu_3
+3\mu_2+6\mu_1^2)}} & \frac{\sqrt{18}\mu_{1,2}}{\sqrt{\mu_{1,2,3}(\mu_3
+3\mu_2+6\mu_1)}} & \frac{\sqrt{2}\mu_{1,3}}{\sqrt{\mu_{1,2,3}(\mu_3
+3\mu_2+6\mu_1)}} \\
\frac{\sqrt{18}\mu_{1,2}}{\sqrt{\mu_{1,2,3}(\mu_3
+3\mu_2+6\mu_1)}}  & \frac{3\mu_2 ^2}{\sqrt{\mu_{1,2,3}(2\mu_3
+6\mu_2+12\mu_1)}} &\frac{\mu_{2,3}}{\sqrt{\mu_{1,2,3}(2\mu_3+6\mu_2+12\mu_1)}}\\
\frac{\sqrt{2}\mu_{1,3}}{\sqrt{\mu_{1,2,3}(\mu_3
+3\mu_2+6\mu_1)}} & \frac{\mu_{2,3}}{\sqrt{\mu_{1,2,3}(2\mu_3+6\mu_2+12\mu_1)}} & \frac{\frac{1}{3}\mu_3 ^2}{\sqrt{\mu_{1,2,3}(2\mu_3+6\mu_2+12\mu_1)}}   \\
\end{array}\right) .   
\end{align*}

\end{exa}

\subsection{Subcritical regime}

We now consider $\Sigma_n$ in the subcritical regime (i.e., $t\delta_t^d \to 0$) for $n \geq 2$. Recall that, according to (\ref{Eq: Kov.matrix Laengen Potenz Funktional}), this matrix is defined as

\[
 \Sigma_n
=A_0^{<1} = \Bigl(\mu_{k_j}^{(\alpha_l+ \alpha_j)} \frac{\mathds{1}(k_l =k_j)}{(k_j + 1)!} 
 \Bigr)\in  \R^{n \times n},
\]

which is a diagonal block matrix, where a diagonal block with size $i>1$ arises whenever there exists an $r$ such that $k_r = k_{r+1}= \ldots =k_{r+i-1}$. These blocks are multiples of generalized $i \times i$-moment matrices
\begin{align} \label{Eq: Ekr}
   E_{k_r} = (\mathbb{E}\Delta_{1,k_r} ^{(\alpha_l+\alpha_j)})_{l,j\in \{r , \ldots, r+i-1\}}, 
\end{align}

where we abbreviate $\Delta_{1,k_r}= \Delta_{1,k_r} [0, \{ X_l\}^{k_r}_{l=1}]$. Due to \cite[Theorem 2.3]{AkinwandeReitzner} each matrix $E_{k_r}$ has full rank. Consequently, $\Sigma_n$ is also of full rank.
Based on these preliminary considerations, we state the following theorem, which also follows directly from \cite[Corollary 3.6]{AkinwandeReitzner}.

\begin{thm} \label{Thm: rank, det, definitness subcrit}
Let $\Sigma_n $ be as in (\ref{Eq: Kov.matrix Laengen Potenz Funktional}) in the subcritical regime for $n\geq2$. Then 
(i) $\emph{rank}(\Sigma_n) = n$;
(ii) $\Sigma_n$ is positive definite;
(iii) $\det \Sigma_n>0$.
\end{thm}

Further results concerning the invariants and decompositions of $\Sigma_n$ in the subcritical regime are derived by examining special cases of the admissible sequence.

\begin{defi}
Let $(k_1, \alpha_1), \ldots,(k_1, \alpha_{l_1}), \ldots,  (k_r, \alpha_{\sum_{h=1}^{r-1} l_h}), \ldots,( k_r, \alpha_n)$ be an admissible sequence for $r \leq n$ such that $k_i \neq k_j$ for $i \neq j \in [r]$ and $l_1, \ldots, l_r \in \mathbb{N}$ such that $\sum _{i=1}^rl_i= n$. Furthermore, let $ s_1, \ldots, s_r \in \mathbb{N}$ be such that the following condition is satisfied: 
\[ | \alpha_j-\alpha_{j+1}|= s_{i} \, \text{ for } \, i \in [r]\, \text{ and }\, j \in \{\sum_{h= 1 }^{i-1}l_{h}+1, \ldots, \sum_{h= 1 }^{i}l_{h}-1\}. \] 
The vector \[(\tilde{\V}_{k_1} ^{(\alpha_{1})}, \ldots, \tilde{\V}_{k_1} ^{(\alpha_{l_1})}, \tilde{\V}_{k_2} ^{(\alpha_{l_1+1})},\ldots, \tilde{\V}_{k_2} ^{(\alpha_{l_1+l_2})}, \ldots,\tilde{\V}_{k_r} ^{(\alpha_{\sum_{i=1}^{r-1}l_i+1})}, \ldots , \tilde{\V}_{k_r} ^{(\alpha_{\sum_{i=1}^{r}l_i})})\] is called the \textit{vector of identical distances}.
\end{defi}

For the vector of identical distances the matrices $E_{k_i}$ are so-called \textit{Hankel matrices} (see, e.g., \cite{Widom}). Several interesting results have already been established for Hankel matrices, including their inverses, eigenvalues, and eigenvectors, which lead immediately to the next theorem. Strategies for computing the determinants of Hankel matrices also exist (see, e.g., \cite{Junod}). These will not be discussed in detail here, but may be of interest for future work.

\begin{thm} \label{Thm: eigenvalues subcritical}
Let $\Sigma_n $ be as in (\ref{Eq: Kov.matrix Laengen Potenz Funktional}) in the subcritical regime for $n\geq2$ with respect to the vector 
of identical distances and $E_{k_i}$ the associated block matrices from (\ref{Eq: Ekr}). Then the following holds:

\begin{enumerate}
    \item  \begin{align*}
    \Sigma_n^{-1} =
\left( 
\begin{array}{cccccccc}
H_{k_1} & 0 &\ldots & 0\\
0 & \ddots& \ddots  & \vdots \\
\vdots & \ddots & \ddots & 0 \\
0 & \ldots &  0 & H_{k_r} 
\end{array}
\right) ,
\end{align*}

where $H_{k_i} = E_{k_i} ^{-1}$ are the inverse matrices  of the $E_{k_i}$ according to the algorithm from \cite[3]{Trench}.
    \item The eigenvalues of $\Sigma_n$ are $ \lambda_{k_1,1}, \ldots, \lambda_{k_1,l_1}, \ldots, \lambda_{k_r,1}, \ldots, \lambda_{k_r,l_r} $, where $\lambda_{k_i, j}$ for $j \in [l_i]$ are the eigenvalues of the Hankel matrices $E_{k_i}$ according to the algorithm from \cite[5]{Luk}.
\end{enumerate}

\end{thm}

\begin{proof}
    Both statements follow directly from the fact that $\Sigma_n$ is a diagonal block matrix. The inverse of a Hankel matrix $E_{k_i}$, where $l_i=n$, can be computed using the algorithm from \cite{Trench} in $O(n^2)$. Additionally, there exists an algorithm in \cite{Luk} that computes the explicit eigenvalues of a Hankel matrix, with a complexity of $O(n^2\log n)$. 
\end{proof} 
\begin{exa}

Let $\Sigma_5$ be  as in (\ref{Eq: Kov.matrix Laengen Potenz Funktional}) in the subcritical regime for an admissible sequence with $d=3$, $k_1=1, k_2=2, k_3 = 3$ and $l_1 = 2, l_2 = 2, l_3 = 1$ and with the powers $ \alpha_1 = 1, \alpha_2 = 3, \alpha_ 3 =1, \alpha_4 = 5, \alpha_5 = 4$. This means that we have
\begin{align*}
\label{Eq: Kov.matrix Laengen Potenz Funktional supercritical}
\Sigma_5
&=
\left( 
\begin{array}{cccccc}
\mathbb{E}\Delta_{1,k_1} ^{(2)} & \mathbb{E}\Delta_{1,k_1} ^{(4)} & 0 & 0&0 \\
\mathbb{E}\Delta_{1,k_1} ^{(4)}  & \mathbb{E}\Delta_{1,k_1} ^{(6)}  &0  &0&0 \\
0& 0 & \mathbb{E}\Delta_{1,k_2} ^{(2)} & \mathbb{E}\Delta_{1,k_2} ^{(6)} & 0\\
0& 0& \mathbb{E}\Delta_{1,k_2} ^{(6)}  & \mathbb{E}\Delta_{1,k_2} ^{(10)} & 0\\
0 & 0 &0& 0& \mathbb{E}\Delta_{1,k_3} ^{(8)}  
\end{array}
\right) .
\end{align*}
The diagonal block matrices in this example are Hankel matrices. The inverse $\Sigma_5^{-1}$ of $\Sigma_5$ is given by

\begin{align*}
\left( 
\begin{array}{cccccc}
S_1 &0 &0 \\
0 & S_2 & 0 \\
0 & 0& S_3
\end{array}
\right),
\end{align*}
where 
\begin{align*}
    S_1 &= \left( 
\begin{array}{cccccc}
\frac{\mathbb{E}\Delta_{1,k_1} ^{(6)} }{\mathbb{E}\Delta_{1,k_1} ^{(2)} \mathbb{E}\Delta_{1,k_1} ^{(6)}- (\mathbb{E}\Delta_{1,k_1} ^{(4)})^2  } & \frac{-\mathbb{E}\Delta_{1,k_1} ^{(4)} }{\mathbb{E}\Delta_{1,k_1} ^{(2)} \mathbb{E}\Delta_{1,k_1} ^{(6)}- (\mathbb{E}\Delta_{1,k_1} ^{(4)})^2  }  \\ \frac{-\mathbb{E}\Delta_{1,k_1} ^{(4)} }{  \mathbb{E}\Delta_{1,k_1} ^{(2)} \mathbb{E}\Delta_{1,k_1} ^{(6)} - (\mathbb{E}\Delta_{1,k_1} ^{(4)})^2 } &
\frac{-\mathbb{E}\Delta_{1,k_1} ^{(2)} }{\mathbb{E}\Delta_{1,k_1} ^{(2)} \mathbb{E}\Delta_{1,k_1} ^{(6)}- (\mathbb{E}\Delta_{1,k_1} ^{(4)})^2 } & \end{array}\right), \\
    S_2 & = \left( 
\begin{array}{cccccc}\frac{\mathbb{E}\Delta_{1,k_2} ^{(10)} }{\mathbb{E}\Delta_{1,k_2} ^{(2)} \mathbb{E}\Delta_{1,k_2} ^{(10)}- (\mathbb{E}\Delta_{1,k_2} ^{(6)})^2  }  & \frac{-\mathbb{E}\Delta_{1,k_2} ^{(6)} }{\mathbb{E}\Delta_{1,k_2} ^{(2)} \mathbb{E}\Delta_{1,k_2} ^{(10)}- (\mathbb{E}\Delta_{1,k_2} ^{(6)})^2  } \\
 \frac{-\mathbb{E}\Delta_{1,k_2} ^{(6)} }{\mathbb{E}\Delta_{1,k_2} ^{(2)} \mathbb{E}\Delta_{1,k_2} ^{(10)} - (\mathbb{E}\Delta_{1,k_2} ^{(6)})^2 }  & \frac{\mathbb{E}\Delta_{1,k_2} ^{(2)} }{\mathbb{E}\Delta_{1,k_2} ^{(2)} \mathbb{E}\Delta_{1,k_2} ^{(10)}- (\mathbb{E}\Delta_{1,k_2} ^{(6)})^2  } \end{array}\right), \\
    S_3 &=  \frac{1}{\mathbb{E}\Delta_{2,k_3} ^{(8)} }.
\end{align*}

\end{exa}

Moreover, we consider a particularly interesting admissible sequence where all $k_i$ are distinct. In this case the matrix $\Sigma_n$ is diagonal.

\begin{defi}\label{Def: vector distinct ki} Let $(k_1, \alpha_1), \ldots ,(k_n, \alpha_n)$ be an admissible sequence with $k_i \neq k_j$ for $i \neq j$ and $ i, j \in [n]$. The vector $(\tilde{\V}_{k_1}^{(\alpha_1)},  \ldots,\tilde{\V}_{k_n}^{(\alpha_n)})$ is called the \textit{vector of distinct} $k_i$\textit{-simplices}.
\end{defi}
Owing to the diagonal structure of $\Sigma_n$ for the vector of distinct $k_i$-simplices, we can directly derive the following results. 

\begin{thm} \label{Thm: eigenvalues distinct k_i}
Let $\Sigma_n $ be as in (\ref{Eq: Kov.matrix Laengen Potenz Funktional}) in the subcritical regime for $n\geq2$ with respect to the vector of distinct $k_i$-simplices 
$(\tilde{\V}_{k_1}^{(\alpha_1)},  \ldots,\tilde{\V}_{k_n}^{(\alpha_n)})$. Then the following statements hold:
\begin{enumerate}
    \item The determinant of $\Sigma_n$ is \[\det \Sigma_n=\prod_{i=1}^n \frac{\mu_{k_i}^{(2\alpha_i)}}{ (k_i + 1)!}  
.\]
    \item The eigenvalues of $\Sigma_n$ are its diagonal entries 
\[\lambda_i = \frac{\mu_{k_i}^{(2\alpha_i)}}{ (k_i + 1)!} \,\text{ for } \,i \in [n].\]

    \item The eigenspaces of $\Sigma_n$ are given by \[\text{eig}(\lambda_i; \Sigma_n) = e_i, \]
    where $e_i$ denotes the $i$-th $n$-dimensional unit vector.
\end{enumerate}
\end{thm}
\begin{proof}
  The results arise directly from the straightforward form of $\Sigma_n$.
\end{proof}
Moreover, the following theorem provides several decompositions of interest.

\begin{thm} \label{thm: compositions subcritical}
Let $\Sigma_n $ be as in (\ref{Eq: Kov.matrix Laengen Potenz Funktional}) in the subcritical regime for $n\geq2$ with respect to the vector of distinct $k_i$-simplices. Accordingly, the following statements hold: 
\begin{enumerate}
    \item $(\Sigma_n)^{1/2}\cdot (\Sigma_n)^{1/2}$ is a Cholesky decomposition of $\Sigma_n$. Additionally, $(\Sigma_n)^{1/2}$ is a matrix root of $\Sigma_n$.
    \item $I_n \cdot \Sigma_n$ is an LU decomposition of $\Sigma_n$, where $I_n$ denotes the $n \times n$-unit matrix.
\end{enumerate}

\end{thm}

\begin{proof}
As before, the results emerge directly from the straightforward form of $\Sigma_n$.
\end{proof}

\begin{rem}
Additionally, there are several other simple cases that we can characterize.
   \begin{enumerate}
       \item The case of an admissible sequence with equal powers $\alpha_i$ is straightforward to analyze. Since all pairs $(k_i, \alpha_i)$ in an admissible sequence must be distinct, this scenario constitutes a special case of the vector of distinct $k_i$-simplices. Consequently, Theorems \ref{Thm: eigenvalues distinct k_i} and \ref{thm: compositions subcritical} remain applicable here.
       \item Furthermore, consider the special case of an admissible sequence $(k_i, \alpha_i)$ where all $k_i$ are equal. Note that for $n=2$, the resulting matrices are Hankel matrices, analogous to the vector of identical distances. Therefore, the results established in Theorem \ref{Thm: eigenvalues subcritical} remain valid here as well.
   \end{enumerate}
   
\end{rem}

\subsection{Critical regime} \label{Kap: Krit. Reg.}

From this point onward, $\Sigma_n$ is considered in the critical regime (i.e., $t\delta_t^d \rightarrow c \in (0, \infty)$) for $n\geq 2$. Recall that in this regime one has to distinguish the following cases:

\begin{align} \label{Eq: Sigma krit. Regime}
  \Sigma_n
= \left\{
\begin{array}{ll}
\sum_{m=0}^{2k_n} A_m^{<1} c^{m/2}&\textit{ for }\, c  \in (0,1], \\
\sum_{m=0}^{k_n} A_m^{>1} c^{-m}& \textit{ for }\,c  \in [1,\infty). \\
\end{array}
\right.  
\end{align}

Some properties and invariants of $\Sigma_n$ have already been investigated in \cite{AkinwandeReitzner}.
In particular, \cite[Corollary 3.6]{AkinwandeReitzner} provides the basis for the next theorem.
\begin{thm} \label{Thm: rank, det, definitness crit}
Let $\Sigma_n $ be as in (\ref{Eq: Kov.matrix Laengen Potenz Funktional}) in the critical regime for $n\geq2$. Except for finitely many values of $c$ the following statements hold:
(i) $\emph{rank}(\Sigma_n) = n$;
(ii) $\Sigma_n$ is positive definite;
(iii) $\det \Sigma_n>0$.
\end{thm}
Beyond these results, we present an approach to determine bounds for the eigenvalues of $\Sigma_n$ under Requirement \ref{Asum 1} for $c \in [1, \infty)$. 

In this regard, we study $A_{m}^{>1}$ for $m \in \{0, \ldots, k_n\}$ for this case. These matrices are block matrices where only the blocks in the right lower corner are nonzero. This means, that there appear a huge number of zeroes in these matrices for big $m$. The matrices $A_{m}^{>1}$ have the structure

\begin{equation}
\label{Eq: Am Matrizen superkrit}
A_m^{>1}
= 
\left( 
\begin{array}{cccccc}
0 &0 \\
0 & C
\end{array} 
\right) \in \mathbb{R}^{n \times n}
\enspace \text{ with } \enspace C = \left( 
\begin{array}{cccccc}
 \frac{\mu_{k_r,k_r:m+1}^{(\alpha_r,\alpha_m)}}{a_{r,m}^2} & \ldots  & \frac{\mu_{k_n,k_r:m+1}^{(\alpha_{r},\alpha_{n})}}{a_{r,m} a_{n,m}}  \\
\vdots&  \ddots &\vdots \\
 \frac{\mu_{k_r,k_n:m+1}^{(\alpha_r,\alpha_n)}}{a_{r,m}a_{n,m}} & \ldots & \frac{\mu_{k_n,k_n:m+1}^{(\alpha_{n},\alpha_{n})}}{a_{n,m}^2}\\
\end{array}
\right),
\end{equation}

where $ a_{i,m} = \sqrt{(k_i-m)!(m+1)!}$
and $r= \text{argmin}_{i \in [n]} m\leq k_r$.

The subsequent analysis utilizes the following bounds.
\begin{lem} \label{lem: estimation mu ki kj}
  Let $(k_1, \alpha_1), \ldots ,(k_n, \alpha_n)$ be an admissible sequence, $m \in \{0, \ldots, k_n\}$, $r= \text{argmin}_{i \in [n]} m\leq k_r$ and $i,j \in \{r, \ldots, n\}$. The $\mu_{k_i,k_j:m+1}^{(\alpha_i,\alpha_j)}$ from (\ref{Eq: mu ki kj}) are bounded above by
  \begin{align} \label{Eq: bound S}
  \nonumber \mu_{k_i,k_j:m+1}^{(\alpha_i,\alpha_j)} &\leq (\frac{d\kappa_d}{\alpha_i+d}) ^{k_i-m} (\frac{d\kappa_d}{\alpha_i+\alpha_j+d}) ^{m} (\frac{d\kappa_d}{\alpha_j+d}) ^{k_j-m}\\
  &\leq \max_{i, j \in \{r , \ldots, n\}}(\frac{d\kappa_d}{\min\{\alpha_i,\alpha_j, \alpha_i+\alpha_j\}+d)}) ^{k_i+k_j-m}:= S_m.
  \end{align}
  
\end{lem}

\begin{proof}
    First notice that the volumes of all simplices in $\eta_t$ are bounded above by the product of the length of appearing edges in each simplex, which are themselves bounded by $1$, since all points are in the $d$-dimensional unit-ball.
    Thus, we see with the transformation $x_l= r_lu_l$ with $r_l= \| x_l \|, u_l \in S^{d-1}$ for $l \in \{1, \ldots, k_i+k_j+1-m\}$, which has Jacobian $r^{d-1}$, that
    \begin{align*}
     \mu_{k_i,k_j:m}^{(\alpha_i,\alpha_j)} &=    \int_{(B^d)^{k_i+k_j+1-m}}\Delta_1[0,\{x_l\}^{k_i}_{l=1}]^{\alpha_1}\Delta_1[0,\{x_l\}^{k_i+k_j-m+1}_{l=k_i-m+2}]^{\alpha_2} \text{d}x_1 \ldots \text{d}x_{k_i+k_j+1-m}\\ 
     &\leq \int_{(S^{d-1})^{k_i+k_j+m-1}} \int_{0}^1... \int_{0}^1 (r_1 \ldots r_{k_i-m+1})^{\alpha_1}(r_{k_i-m+2}\ldots r_{k_i})^{\alpha_i+\alpha_j}\\
     &(r_{k_i-m+2} \ldots r_{k_j+k_i+1-m})^{\alpha_j} (r_1 \ldots r_{k_i+k_j+1-m})^{d-1} \text{d}r_1 \ldots \text{d}r_{k_i+k_j+1-m} \text{d}u_1 \ldots \text{d}u_{k_i+k_j+1-m} \\
     &\leq (\frac{d\kappa_d}{\alpha_i+d}) ^{k_i-m+1} (\frac{d\kappa_d}{\alpha_i+\alpha_j+d}) ^{m-1} (\frac{d\kappa_d}{\alpha_j+d}) ^{k_j-m+1}.
    \end{align*}

    We bound the $\mu_{k_i,k_j:m}^{(\alpha_i,\alpha_j)}$ for all $i, j \in \{r , \ldots, n\}$ by considering the maximum of the appearing expressions in the first estimation and obtain the bound $S_m$ for $m+1$.
\end{proof}
In the upcoming part, we make conclusions in the case that the following requirement is fulfilled.
For this purpose, the Loewner order \cite{Jarre} is introduced, which is defined as follows: Let $A, B$ be two real-valued symmetric matrices. Then
\[A \geq B \,\text{ if and only if } \,A-B \text{ is positive semi-definite}.\]

\begin{Req} \label{Asum 1}
To establish certain bounds we impose that $A_m^{>1}$ are positive semidefinite and that the following constraints for all matrices $A_m^{>1}$ from (\ref{Eq: Am Matrizen superkrit}) hold:
 \begin{equation}
\label{Eq: DB}
E_m \cdot D_m \geq A_m^{>1}, 
\end{equation}
where $D_m =\left( 
\begin{array}{cccccc}
0 &0 \\
0 & \tilde{C}
\end{array}
\right)\in \mathbb{R}^{n \times n}$ such that $ \tilde{C} = \left( 
\begin{array}{cccccc}
 \frac{1}{a_{r,m}^2} & \ldots  & \frac{1}{a_{r,m} a_{n,m}}  \\
\vdots&  \ddots &\vdots \\
 \frac{1}{a_{r,m}a_{n,m}} & \ldots & \frac{1}{a_{n,m}^2}\\
\end{array}
\right)$ and where $E_m$ are the $n \times n$-diagonal matrices such that all diagonal entries are equal to $S_m$ from (\ref{Eq: bound S}).
\end{Req}

  The conditions stated in Requirement \ref{Asum 1} may not be generally satisfied. Below, the case of $C$ and $\tilde{C}$ being $2 \times 2$-matrices is examined, in which the constraints can be specified. Here, $A_m^{>1}$ is defined as\[ C=\left( 
\begin{array}{cccccc}
 \frac{\mu_{k_r,k_r:m+1}^{(\alpha_r,\alpha_r)}}{a_{r,m}^2} & \frac{\mu_{k_r,k_n:m+1}^{(\alpha_r,\alpha_n)}}{a_{r,m} a_{n,m}}  \\
 \frac{\mu_{k_r,k_n:m+1}^{(\alpha_r,\alpha_n)}}{a_{r,m}a_{n,m}} & \frac{\mu_{k_n,k_n:m+1}^{(\alpha_n,\alpha_n)}}{a_{n,m}^2}\\
\end{array}
\right).\]
Now we have to check whether  
\begin{align} \label{Eq: matrix difference}
   E_mD_m-A^{>1}_m= \left( 
\begin{array}{cccccc}
 0 & 0  \\
 0 & \bar{C}\\
\end{array}
\right)
 \text{ with } 
  \bar{C}=  \left( 
\begin{array}{cccccc}
 \frac{S_m-\mu_{k_r,k_r:m+1}^{(\alpha_r,\alpha_r)}}{a_{r,m}^2} & \frac{S_m-\mu_{k_r,k_n:m+1}^{(\alpha_r,\alpha_n)}}{a_{r,m} a_{n,m}}  \\
 \frac{S_m-\mu_{k_r,k_n:m+1}^{(\alpha_m,\alpha_n)}}{a_{r,m}a_{n,m}} & \frac{S_m-\mu_{k_n,k_n:m+1}^{(\alpha_n,\alpha_n)}}{a_{n,m}^2}\\
\end{array}
\right)
\end{align}
is positive semidefinite. It is sufficient to examine $\bar{C}$. The matrix is positive semidefinite provided that the following condition holds:
\[\frac{S_m-\mu_{k_r,k_r:m+1}^{(\alpha_r,\alpha_r)}}{a_{r,m}^2} \cdot \frac{S_m-\mu_{k_n,k_n:m+1}^{(\alpha_n,\alpha_n)}}{a_{n,m}^2} \geq (\frac{S_m-\mu_{k_r,k_n:m+1}^{(\alpha_r,\alpha_n)}}{a_{r,m} a_{n,m}})^2, \]
which reduces to 
\[(S_m-\mu_{k_r,k_r:m+1}^{(\alpha_r,\alpha_r)})(S_m-\mu_{k_n,k_n:m+1}^{(\alpha_n,\alpha_n)})\geq (S_m-\mu_{k_r,k_n:m+1}^{(\alpha_r,\alpha_n)})^2 . \]
Furthermore, it is necessary to check whether $A_m ^{> 1}$ are positive semidefinite. By virtue of Equation (\ref{Eq: Am Matrizen superkrit}), we can restrict the analysis to $C$.
If $C$ consists only of a single entry, it is immediately apparent that this entry is positive and consequently $A_m ^{> 1}$ is positive semidefinite. If 
$C$ is $2 \times 2$ this means that we can check by a similar way as above, if it's a positive semidefinite matrix (e.g., the case $m=0$ is discussed in the supercritical regime). 

Assuming Requirement \ref{Asum 1} holds, we continue to find bounds for the eigenvalues of $\Sigma_n$ in the critical regime. To do this, apply the following observation concerning the matrices $D_m$ and $E_m$.

\begin{lem} \label{lem: matrix D}
Let $D_m, E_m \in \mathbb{R}^{n \times n}$ be matrices according to (\ref{Eq: DB}) for $m \in \{0, \ldots, k_n\}$. The following statements hold:
\begin{enumerate}
    \item The eigenvalues of $D_m$ are $\lambda_1 =0$ and $\lambda_2 =\sum_{i=r}^n 1/ a_{i,m}^2$, where $a_{i,m}$ are as in (\ref{Eq: Am Matrizen superkrit}).
    \item  The eigenvalues of $E_m$ are all equal to the bound $S_m$ from (\ref{Eq: bound S}).
\end{enumerate}

\end{lem}

\begin{proof}
(i) Similar as in Theorem \ref{Thm: EW.e, ER.e superkrit. Reg.} the results from \cite{ReitznerRoemer} can be used to compute the eigenvalues of the lower right nonzero block matrix $\tilde{C}$.  \\
(ii) This follows directly from the diagonal structure of the $E_m$.
 
\end{proof}

Next, the \textit{spectral norm} $\|\cdot \|_2$ is employed. For $A \in \mathbb{R}^{n\times n}$:
\begin{align}\label{rem: spectral norm}
 \| A \|_2 = \sqrt{\lambda_{\max}}   ,
\end{align}
where $\lambda_{\max}$ is the largest eigenvalue of $A^tA$. If $A$ is symmetric, then $\| A \|_2$ equals the absolute value of the largest eigenvalue of $A$, as $A^tA = A^2 $. The spectral norm exhibits a number of important characteristics. 

\begin{lem} \label{lem: spectral norm}The spectralnorm $\| \cdot \|_2$ satisfies the following statements: \
  \begin{enumerate}
    \item $\| \cdot \|_2$ is compatible with the euclidean norm $\| \cdot \|_2$, i.e., \[\| Ax \|_2 \leq \| A \|_2 \| x \|_2 \text{ for } x \in \mathbb{R}^n.\] 
     \item $\| \cdot \|_2$ is submultiplicative.
     \item Let $A,B \in \mathbb{R}^{n \times n}$ be symmetric and $B$ be positive semidefinite. Then, $A \geq B$ implies that $\| A \|_2 \geq \| B \|_2$.
\end{enumerate}   
\end{lem}

\begin{proof}    
    (i) One can rewrite the spectral norm of a matrix $A\in \mathbb{R}^{n\times n}$ as
    \begin{align} \label{Eq: spectral norm}
        \| A \|_2 = \max_{x \neq 0} \frac{\| Ax\|_2}{\| x\|_2},
    \end{align}
    since for an orthonormal basis of eigenvectors $(x_1, \ldots, x_n)$ with respect to the eigenvalues $\lambda_1, \ldots, \lambda_n$ of $A^tA$ we have for $x \in \mathbb{R} ^n$ that 
    \begin{align*}
        \| Ax\|^2_2 &= \langle Ax, Ax \rangle = \langle A^tAx, x\rangle =  \langle \sum _{i=1}^n\langle x, x_i\rangle A^tAx_i, \sum _{i=1}^n\langle x, x_i\rangle x_i \rangle = \sum_{i=1}^n \lambda_i \langle x, x_i \rangle^2 \\ &\leq \lambda_{\max} \| x \|_2^2.
    \end{align*}
    
    Furthermore, for a normalized eigenvector $\tilde{x}$ of $\lambda_{\max}$ it follows from the above, that $\| A\tilde{x} \|^2 _2= \lambda_{\max}$. This shows that \[\| A \| ^2_2 = \lambda_{\max} = \max_{x \neq 0 }\frac{\| Ax \|^2_2}{\| x \|^2_2} \text{ for }  x \in \mathbb{R}^n \]
    and therefore (\ref{Eq: spectral norm}) holds. This equation implies for $x \in \mathbb{R} ^n$ that
    \[\| A \|_2 \| x\|_2 \geq   \| Ax\|_2,\] i.e., the spectral norm is compatible with the euclidean norm. \\
    (ii) We need to show that $\| AB \|_2 \leq \| A \| _2\| B \|_2$. By (i), it follows that   
    \begin{align*}
     \| A B x\|_2 \leq  \| AB \|_2 \| x \|_2 \text{ for } x \in \mathbb{R}^n.  
    \end{align*}
    Furthermore, $\| A B x\|_2$ can be estimated in the following way:
      \begin{align*}
     \| A( B x)\|_2 \leq  \| A \|_2 \| Bx \| _2 \leq  \| A \|_2 \| B \|_2 \| x \|_2.
     \end{align*}
     Using (\ref{Eq: spectral norm}) one sees that there exists $\tilde{x}$ with $\| \tilde{x} \|_2 \neq 0$ such that 
     
  \begin{align*}
     \| A B  \|_2 \| \tilde{x}\|_2 = \| A B \tilde{x}\|_2  \leq   \| A \|_2 \| B \|_2 \| \tilde{x} \|_2.
     \end{align*} 
     Dividing by $\| \tilde{x} \|$ yields the desired result.

    (iii) 
    If $A \geq B$ than $A-B$ is positive semidefinite. Furthermore, we assume that $B$ is positive semidefinite. 
    These conditions imply, that $|\lambda_i^A| \geq |\lambda_i^B|$ for all eigenvalues $i \in [n]$. We use the following deductions from the Theorem of Courant--Fischer (see \cite[Lemma 5.1(i)]{ReitznerRoemer} for a proof) to show this:
    
    For two symmetric matrices $U, \, V \in \R^{n \times n}$ and $W= U+V\in \R^{n \times n}$ with corresponding eigenvalues 
$\lambda_1^U \leq \ldots\leq \lambda_n^U$, $\lambda_1^V \leq \ldots \leq \lambda_n^V$ and $\lambda_1^W \leq \ldots \leq \lambda_n^W$ holds, that
\begin{align} \label{Eq: Eigenwertschranken}
    \lambda_j^U + \lambda_1^V \leq \lambda_j^W \leq \lambda_j^U + \lambda_n^V \text{ for } j \in [n] .
\end{align} 
    
    Because $A$ is the sum of the positive semidefinite matrices $B$ and $A-B$, we obtain the following property:
    \[\lambda_1^{A-B} + \lambda_i^B\leq \lambda_i^{A}. 
    \]
    Consequently, 
     \[0 \leq \ \lambda_1^{A-B} \leq \lambda_i^{A} - \lambda_i^B 
     \]
     and therefore $|\lambda_i^{A} | \geq |\lambda_i^B |$ for all $ i \in [n]$. Accordingly, it has been shown that
     \[\| B \|_2 \leq \| A \|_2.\]
\end{proof}
 
We can now proceed to bound the eigenvalues of the matrices $A_{m}^{>1}$ using the spectral norm. 

\begin{thm} \label{Thm: bounds Am}
Let $A_{m}^{>1}$ be the matrices according to (\ref{Eq: Am<1}) and let $S_m$ be the bounds from (\ref{Eq: bound S}) for $m \in \{0, \ldots, k_n\}$. Moreover, assume Requirement \ref{Asum 1}. The eigenvalues of $A_{m}^{>1}$ are bounded above by 
   \[ S_m\sum_{i=r_m}^n \frac{1}{a_{i,m}^2},\]
   with $a_{i,m}$ from (\ref{Eq: Am Matrizen superkrit}) and $r_m= \text{argmin}_{i \in [n]} m\leq k_r$.
\end{thm}

\begin{proof} Let $\lambda_i ^{ A_{m}^{>1}}$ be the eigenvalues of $ A_{m}^{>1}$ and $E_m, D_m$ be the matrices from (\ref{Eq: DB}). Using the properties (ii) and (iii) from Lemma \ref{lem: spectral norm} for the spectral norm, as well as Lemma \ref{lem: matrix D} in the last step, shows that

\begin{align*}
    |\lambda_i^{ A_{m}^{>1}}| &\leq \| A_{m}^{>1} \|_2 \leq \| E_m\cdot D_m\|_2 \leq \| E_m\| _2 \| D_m \|_2 =  S_m\sum_{i=r_m}^n \frac{1}{a_{i,m}^2}
\end{align*}
for all $m \in \{0, \ldots, k_n\}$.
\end{proof}

Based on the preceding results, the following conjecture for $\Sigma_n$ in the critical regime is proposed.

\begin{conj} \label{conj: bounds eigenvalues crit reg}
Let $\Sigma_n$ be as in (\ref{Eq: Kov.matrix Laengen Potenz Funktional}) in the critical regime such that the matrices $A_{m}^{>1}$ for $m \in \{0, \ldots, k_n\}$ according to (\ref{Eq: Am<1}) fulfill Requirement \ref{Asum 1}. Furthermore, let $S_m$ be the bounds from (\ref{Eq: bound S}). The eigenvalues $\lambda_1^ {\Sigma_n}\leq \ldots \leq \lambda_n^ {\Sigma_n}$ of $\Sigma_n$ are bounded by
\begin{align*}
  0 \leq  \lambda_1^ {\Sigma_n}\leq \ldots \leq \lambda_n^ {\Sigma_n} \leq  (k_n+1) \max_{m=0, \ldots, k_n}S_m(\sum_{i=r_m}^n \frac{1}{ a_{i,m}^2}),  
\end{align*}
where $r_m= \text{argmin}_{i \in [n]} m\leq k_r$.
\end{conj}

\begin{rem}
Assuming the requirements of Conjecture \ref{conj: bounds eigenvalues crit reg} hold, this statement can be proven.
For this, we denote by $\lambda_1^{A_m^{>1}} \leq \ldots \leq \lambda_n^{A_m^{>1}} $ the eigenvalues of the matrices $A_m^{>1}$ in arising order. These eigenvalues are themselves bounded above as given in Theorem \ref{Thm: bounds Am}. We use the right part of Equation (\ref{Eq: Eigenwertschranken}). Based on these preliminary considerations and the fact that $c^{-m}$ 
is bounded by 1, the case 
$c \in [1, \infty)$ satisfies 
\begin{align*}
  0 \leq  \lambda_1^ {\Sigma_n}\leq \ldots \leq \lambda_n^ {\Sigma_n} \leq \sum_{m=0} ^{k_n} \lambda_n^{A_m^{>1}} c^{-m}\leq 
 (k_n+1) \max_{m=0, \ldots, k_n}S_m(\sum_{i=r_m}^n \frac{1}{ a_{i,m}^2}).  \enspace \enspace \enspace \enspace \qedhere
\end{align*}

\end{rem}

\section{Outlook} \label{Section: Outlook}
In this paper we study a vector of volume power functionals regarding important algebraic invariants and aspects in all arising regimes. In particular, the results of \cite{ReitznerRoemer} concerning length power functionals are extended to a more general setting, supplemented by algebraic results for volume power functionals from \cite{AkinwandeReitzner}. \\
Tables \ref{Tab: first results in all regimes} and \ref{Tab: decompositions in all regimes} in Section \ref{Kap: Main results} give an overview of our results and show, that there are still a lot of interesting open questions for future work. For example, there are certain aspects and regimes where one can work on more general solutions (e.g., where we consider in this work only special cases) or more explicit statements (e.g., where we state in this work only bounds). Additionally, Section \ref{Sec: stoch. appl.} discusses stochastic ideas and approaches, highlighting how the algebraic results provide advantages or insights for volume power vectors. Furthermore, a discussion on $h$-vectors is initiated in that section. A more detailed investigation of the $h$-vector is intended for future work. This vector is, in particular, important in combinatorics (see, e.g., \cite{Bhaskara, McMullen}) and its relationship to volume power functionals was emphasized above.

The numerical results presented in Figure \ref{fig:sensitivity} provide critical insights into the stability of the supercritical regime. We observe that the relative sensitivity $\mathcal{R}_{ij}$ is not a constant, but a strictly decreasing function of $\alpha$.
This leads to the following conclusions:
A compelling direction for future work is the detailed analysis of the phase transition between the subcritical and supercritical regimes. The supercritical regime is characterized by a rank-one covariance matrix $\Sigma_n$, which leads to asymptotic redundancy. But the manner by which the system transitions from a positive-definite matrix (subcritical) to a rank-one structure (supercritical) is not yet fully understood.
Future investigations should aim to characterize the critical point of this transition and focus on the investigation of the progression of the eigenvalue spectrum. This would not only complete the theoretical setting of volume power functionals but also allow for the optimization of dimensionality reduction techniques in stochastic geometry. \\

During the preparation of this work, the author used Kiwi secure (a tool from Osnabrueck University) to generate the Python script. The author reviewed and edited the output as needed and takes full responsibility for the content of the published article.

\end{document}